\documentclass{elsarticle}
\usepackage[utf8]{inputenc}
\usepackage[english]{babel}
\usepackage{amsmath,amssymb,amsfonts,amsthm}
\usepackage[hidelinks,colorlinks = true]{hyperref}
\usepackage{enumitem}
\usepackage{tikz-cd}
\usepackage{tcolorbox}
\newtcolorbox{mybox}{colback=red!5!white,colframe=red!75!black, sharp corners = all}
\usepackage{bclogo} 
\def\ucsign{\bcpanchant}

\let\det\undefined
\DeclareMathOperator{\det}{det}
\DeclareMathOperator{\rk}{rk}

\DeclareMathOperator{\id}{id}

\DeclareMathOperator{\Img}{im}
\DeclareMathOperator{\Char}{char}
\DeclareMathOperator{\Span}{span}

\DeclareMathOperator{\rad}{rad}

\DeclareMathOperator{\colsp}{\mathbf C}

\def\F{{\mathbb F}}
\def\L{{\mathcal L}}
\def\M{{\mathcal M}}
\def\SCS{{\mathbf S}}

\newenvironment{psmallmatrix}
  {\left(\begin{smallmatrix}}
  {\end{smallmatrix}\right)}

\newtheorem{theorem}{Theorem}[section]
\newtheorem{lemma}[theorem]{Lemma}
\newtheorem{corollary}[theorem]{Corollary}

\theoremstyle{definition}
\newtheorem{definition}[theorem]{Definition}

\newtheorem{remark}[theorem]{Remark}
\newtheorem{example}[theorem]{Example}
\numberwithin{equation}{section}

\begin{document}
\title{Nonlinear maps preserving the polynomial}
\author[1]{Andrey Yurkov}
\ead{andrey.yurkov@biu.ac.il}
\affiliation[1]{organization={
Bar-Ilan University},
postcode={5290002},
city={Ramat-Gan},
country={Israel}}

\begin{abstract}
Let $\F$ be a field and $P \in \F[x_1,\ldots, x_n]$ be a homogeneous polynomial such that $|\F| > \deg(P)$ and $\phi, \psi\colon \F^n \to \F^n$ be two maps such that $P(\mathbf{x} + \lambda\mathbf{y}) = P(\phi(\mathbf{x}) + \lambda \psi (\mathbf{y}))$ for all $\lambda \in \F$ and $\mathbf{x}, \mathbf{y} \in \F^n.$

We provide the characterization of all such $\phi$ and $\psi$ for all polynomials in the case if $\Char(\F) = 0$ and for all polynomials satisfying certain condition in the case if $\Char(\F) > 0$. This characterization generalizes the existing results regarding the linear maps on matrices
preserving the determinant, the immanant and other homogeneous polynomial functions of matrix entries.

To obtain the main result of this paper, we introduce the vector space $\mathcal L_{P} \subseteq {\F^n}^*$ spanned by the range of the gradient field of $P \in \F[x_1,\ldots, x_n]$. Being a linear invariant associated with $P,$ this space has several remarkable properties and may also be used for studying the linear maps preserving $P$.

In addition, we demonstrate how the main result could be applied to the particular polynomial matrix invariants. Namely, we provide an explicit description of corresponding pairs of nonlinear maps $\phi, \psi$  for the case where $P$ is equal to the Cullis' determinant of $n\times k$ rectangular matrix (with the assumption that $n \ge k + 2$ and $k \ge 3$).
\end{abstract}


\maketitle

\thispagestyle{empty}

\section{Introduction}
\label{sec:intro}
The investigations of linear maps between the matrix spaces that preserve different matrix invariants go back to Frobenius and his theorem about linear maps preserving matrix determinant published in~1897 (see~\cite{GF}). 

\begin{theorem}[Frobenius, {\cite[\textsection 7, Theorem I]{GF}}]\label{thm:frobeniusthm}
Let $S\colon \M_{n}(\mathbb C) \to \M_{n}(\mathbb C)$ be a bijective linear map satisfying $\det (S(X))=\det (X)$ for all $X \in \M_{n} (\mathbb C)$, where $\mathbb C$ denotes the field of complex numbers and $\M_{n} (\mathbb C)$ denotes the space of all square matrices of size $n$ with entries from $\mathbb C$. Then there exist matrices $M, N \in \M_n(\mathbb C)$ with $\det (MN) = 1$ such that
$$S(X) = MXN\;\;\mbox{for all}\; X \in \M_{n} (\mathbb C)
\;\;\text{or}\;\;
S(X) = MX^{t}N\;\;\mbox{for all}\; X \in \M_{n}(\mathbb C).$$
\end{theorem}

This field of mathematics continues to be a subject of active research and has applications in  operator theory, functional analysis and quantum information theory. We refer the reader to~\cite{LAMA199233} for the comprehensive survey  of results up to the end of the 20th century.

In the recent works (\cite{Dolinar2002,TAN2003311,Kuzma2008,Costara2021} and others) the linearity of the considered maps has been replaced to a slightly weaker condition. In 2002, Dolinar and \v{S}emrl provided a characterisation of the surjective maps $\phi \colon \M_n(\mathbb C) \to \M_n(\mathbb C)$ on the space of all square matrices of size $n$ such that
\begin{equation}\label{eq:conditroduction1}
\det (A + \lambda B) = \det (\phi (A) + \lambda \phi (B))\;\;\mbox{for all}\;\;A, B \in \M_n(\mathbb C)\;\;\mbox{and}\;\;\lambda \in \mathbb C
\end{equation}
(see~\cite{Dolinar2002}). Namely, they established that every surjective map $\phi \colon \M_n(\mathbb C) \to \M_n(\mathbb C)$ satisfying the condition~\eqref{eq:conditroduction1} is linear. Then, Theorem~\ref{thm:frobeniusthm} is consequently applied to obtain the required characterization. Soon after, in 2003, Tan and Wang showed that the surjectivity of $\phi$ could be omitted and the field $\mathbb C$ of complex numbers could be replaced to an arbitrary field of sufficient cardinality (see~\cite{TAN2003311}).

In 2008, Kuzma provided the characterization of $\phi$ satisfying the condition~\eqref{eq:conditroduction1}, where the determinant is replaced to the matrix immanant (see~\cite{Kuzma2008}). Moreover, he showed that the map $\phi$ is linear even if the immanants on the left-hand side and the right-hand side of this condition are not equal (for example, if we put the matrix permanent on the left-hand side, and the matrix determinant on the right-hand side).

In 2021, Costara provided another extension of the results of Dolinar, \v{S}emrl, Tan and Wang discussed above (see~\cite{Costara2021}). He considered the pairs of maps $\phi$ and $\psi$ on $\M_{n}(\mathbb C)$ such that
\begin{equation}\label{eq:condintroduction2}
E_k (A + \lambda B) = E_k (\phi (A) + \lambda \psi (B))\;\;\mbox{for all}\;\;A, B \in M_{n}(\mathbb C)\;\;\mbox{and}\;\;\lambda \in \mathbb C.
\end{equation}
 and one of the maps $\phi, \psi$ is surjective. In this condition by $E_k(X)$ we denote the coefficient of $t^{n-k}$ in the of expansion the characteristic polynomial of $X,$ i.e., the polynomial $\det (tI_n + X).$ He established that if $k \ge 2$, $\phi$ and $\psi$ satisfy the condition~\eqref{eq:condintroduction2}, then $\phi = \psi$ and both $\phi$ and $\psi$ are linear. Since $E_{n}(X) = \det (X)$ for all $X \in \M_{n}(\mathbb C)$, this fact indeed extends the results from~\cite{Dolinar2002, TAN2003311} and~\cite{Costara2021}.

In this paper we make the further steps in this direction. Namely, we show that the Costara's result holds for all homogeneous polynomial functions on $\F^n$, where $\F$ denotes a field with zero characteristic, and omit the condition of surjectivity as it is formulated in the theorem below.

\begin{theorem}[see Theorem~\ref{thm:homogenphipsichar0}]\label{thm:homogenphipsiintrochar0}Assume that $\Char(\F) = 0$. Let $P \in \F[x_1,\ldots, x_n]$ be a homogeneous polynomial.  If  $\phi, \psi \colon \F^n \to \F^n$ are such that
\begin{equation*}
P(\mathbf{x} + \lambda \mathbf{y}) = P(\phi(\mathbf{x}) + \lambda \psi(\mathbf{y}))\;\;\mbox{for all}\;\;\mathbf{x},\mathbf{y} \in \mathbb \F^n,\; \lambda \in \F,
\end{equation*}
then there exists a unique linear map $T_{\rad} \colon \F^n/\rad(P) \to \F^n/\rad(P)$ preserving $P_{\rad(P)}$ such that
\begin{equation*}
\pi_{\rad(P)} \circ \phi = \pi_{\rad(P)} \circ \psi = T_{\rad(P)} \circ \pi_{\rad(P)},
\end{equation*}
where by $\rad(P)$ we denote the radical of $P$ defined by $$\rad(P) = \{\mathbf x \in \F^n \mid P(\mathbf a + \lambda \mathbf x) = P(\mathbf a)\;\;\mbox{for all}\;\; \mathbf a \in \F^n, \lambda \in \F\}.$$
\end{theorem} 
The notion of the radical of the function (Definition~\ref{def:radical}) was introduced by Waterhouse in~\cite{Waterhouse1983}. This notion plays an important role in the theory of linear maps preserving matrix invariants. For example, the radical of the function $f$ is equal to zero if and only if every linear map preserving $f$ is invertible (Proposition~1 in~\cite{Waterhouse1983}).  
Accordingly, it is possible to show that $\rad(E_k) = \{0\}$ for $k \ge 2$ (i.e.,~\cite[Corollary~3.2]{Costara2021}), which implies that Theorem~\ref{thm:homogenphipsiintrochar0} indeed extends the result of Costara. 

Furthermore, the statement of Theorem~\ref{thm:homogenphipsiintrochar0} could be generalized for the fields $\F$ of positive characteristic and every polynomial $P \in \F[x_1,\ldots, x_n]$ satisfying certain restriction. This comprises the main result of this paper. In order to formulate it we introduce the notion of vector space $\mathcal L_P \subseteq \left(\F^n\right)^*$ associated with $P$ (here $\left(\F^n\right)^*$ denotes the dual vector space). Namely, by $\mathcal L_P$ we denote the vector space generated by all linear functions on $\F^n$ having the form $\mathbf v \mapsto \frac{\partial P}{\partial \mathbf v}(\mathbf a)$ for some fixed $\mathbf a \in \F^n$ (Definition~\ref{def:LP}). We also need to notice that the definition of the radical implies that $\rad(P)$ forms a vector space. Thus, the main theorem of this paper is formulated as follows.

\begin{theorem}[see Theorem~\ref{thm:homogenphipsi}]\label{thm:homogenphipsiintro}Let $P \in \F[x_1,\ldots, x_n]$ be a homogeneous polynomial such that $\deg (P) < |\F|$ and $\dim (\mathcal L_P) + \dim (\rad(P)) = n$.  If  $\phi, \psi \colon \F^n \to \F^n$ are such that
\begin{equation}\label{thm:homogenphipsiintro:eqq1}
P(\mathbf{x} + \lambda \mathbf{y}) = P(\phi(\mathbf{x}) + \lambda \psi(\mathbf{y}))\;\;\mbox{for all}\;\;\mathbf{x},\mathbf{y} \in \mathbb \F^n,\; \lambda \in \F,
\end{equation}
then there exists a unique linear map $T_{\rad} \colon \F^n/\rad(P) \to \F^n/\rad(P)$ preserving $P_{\rad(P)}$ such that
\begin{equation*}
\pi_{\rad(P)} \circ \phi = \pi_{\rad(P)} \circ \psi = T_{\rad(P)} \circ \pi_{\rad(P)}. 
\end{equation*}
\end{theorem} 

To establish Theorem~\ref{thm:homogenphipsiintro}, we adopt the beautiful argument of Tan and Wang to the case where $\rad(P)$ could be arbitrary and show that $\psi$ is a bijective linear map when it is considered modulo $\rad(P)$ (see Lemma~\ref{lem:PPhiPsi}). The proof employs the properties of $\mathcal L_P$ which are studied in Section~\ref{sec:linspace}. In particular, we show that $\rad(P)$ belongs to the intersection of kernels of all elements of $\mathcal L_P$ (see Lemma~\ref{lem:RadAnnLP}) and the converse inclusion holds if and only if $\dim (\mathcal L_P) + \dim (\rad(P)) = n$ (see Lemma~\ref{lem:dzeroeqk}). Then, using the homogeneity of $P$, we establish the required proposition. In addition, the properties of $\mathcal L_P$ discussed in Section~\ref{sec:linspace} could be applied for finding $\rad(P)$. Consequently, the notion of $\mathcal L_P$ is useful in the theory of linear maps preserving matrix invariants (see Remark~\ref{rem:LPUseful}).

%
%
%

Accordingly, the condition~\eqref{thm:homogenphipsiintro:eqq1} is restrictive enough to imply that the considered maps are linear modulo $\rad(P)$. In general, for a given polynomial $P \in \F[x_1,\ldots, x_n]$, the condition $P(\phi(\mathbf x)) = P(\mathbf x)$ for all $\mathbf x \in \F^n$ does not imply that $\phi$ is linear modulo $\rad(P)$ (consider for example the case when $\F = \mathbb R$, $P = x^2$, $\phi(x) = |x|$). On the other hand, in certain cases it is possible to characterize nonlinear maps preserving the zero set of certain matrix polynomials. For instance, in~\cite{SEMRL20081051} the corresponding characterisation is provided for commutativity preserving continuous maps $\phi$ on the set of square complex matrices. Furthermore, Guterman and Kuzma in~\cite{Guterman11112009} obtained the description of maps preserving the fixed polynomial belonging to a wide class of matrix polynomials without any preliminary assumption on $\phi$.

%
%
%

Beyond obtaining general results discussed above, we demonstrate in Section~\ref{sec:AppToCullisDet} how they are applied in order to extend the Costara's theorem to the particular polynomial matrix invariants, namely to the Cullis' determinant (Theorem~\ref{thm:homogenphipsiDetNKEven} and Theorem~\ref{thm:homogenphipsiDetNKOdd}). The Cullis' determinant was introduced by Cullis in~1913 in his monograph~\cite{cullis1913}. It is denoted by $\det_{n\,k}$ and extends the notion of the ordinary determinant for rectangular matrices. That is, if $X \in \M_{n\,k}(\F)$, then $\det_{n\,k}(X)$ is defined as an alternating sum of the maximal minors of $X$ (Definition~\ref{def:DETNK}). The complete description of linear maps preserving determinant for the case if $|\F| > \deg(\det_{n\,k}) = k$ has been obtained recently in~\cite{Guterman2025,Guterman2025b,Guterman2025c}. Note that $\rad(\det_{n\,k})$ is nonzero if $n + k$ is odd (Lemma~\ref{lem:NKOddRadDetNK}) which implies that the maps $\phi$ and $\psi$ in~\eqref{thm:homogenphipsiintro:eqq1} could be nonsurjective (see Proposition~1 in~\cite{Waterhouse1983} for the case if $\phi$ and $\psi$ are linear and $\phi = \psi$).

This paper is organized as follows: in Section~\ref{sec:notation} we introduce the notation used throughout this paper; in Section~\ref{sec:prelim} we provide the preliminary facts from algebra and the theory of radical of a function; in Section~\ref{sec:linspace} we introduce the vector space $\mathcal L_P$ associated with every polynomial $P \in \F[x_1,\ldots, x_n]$ and prove its properties; in Section~\ref{sec:factspres} we study the properties of pairs of maps satisfying the condition~\eqref{thm:homogenphipsiintro:eqq1} and prove the main theorem; in Section~\ref{sec:AppToCullisDet} we obtain the description of pairs of maps satisfying the condition~\eqref{thm:homogenphipsiintro:eqq1} for the case if\linebreak $P = \det_{n\,k}$. Finally, Section~\ref{sec:furthwork} contains our comments about the possible further work.

\section{Notation and basic definitions}
\label{sec:notation}
By $\F$ we denote a field without any restrictions on its characteristic and cardinality. 

We denote by $\M_{n\, k}(\F)$ the set of all $n\times k$ matrices with the entries from the field $\F.$ 
$I_{n\, n} = I_{n} \in \M_{n\, n}(\F)$ denotes an identity matrix. By $E_{ij} \in \M_{n\, k}(\mathbb F)$ we denote a matrix whose entries are all equal to zero besides the entry on the intersection of the $i$-th row and the $j$-th column, which is equal to one. By $x_{i\, j}$ we denote the element of a matrix $X$ lying on the intersection of its $i$-th row and $j$-th column (the same convention is assumed for the matrices denoted by other letters, which are always latin capitals).  For integers $i, j$ we denote by $\delta_{i\,j}$  a \emph{Kronecker delta} of $i$ and $j$, which is equal to $1$ if $i = j$ and equal to $0$ otherwise. If $A \in \M_{n\,k}(\F),$ then by  $\colsp (A) \subseteq \F^n$ we denote a \emph{column space of} $A$.

We use the bold font to denote vectors and lower indices to denote their coordinates. In the case if we need the series of vectors, we use the upper indices placed in braces. For example, if $\mathbf v = \begin{pmatrix}1 \\ 0\end{pmatrix}$, then $\mathbf v^t = \begin{pmatrix}1 & 0\end{pmatrix}$ and $\mathbf v_1 = 1$. If $\mathbf u^{(1)} = \begin{pmatrix}1 \\ 0\end{pmatrix}, \mathbf u^{(2)} =  \begin{pmatrix}0 \\ 1\end{pmatrix}$, then $\mathbf u^{(1)}_1 = 1$ and $\mathbf u^{(2)}_1 = 0$.

\section{Preliminaries}
\label{sec:prelim}

Recall the following general facts regarding polynomials over a field.

\begin{lemma}[{Cf.~\cite[IV, \S 1, Corollary 1.7.]{Lang2002}}]Let $\F$ be an infinite field. Let $f$ be a
polynomial in $n$ variables over $\F$. If $f$ defines the zero function on $\F^{n}$, then $f = 0$.
\end{lemma}

\begin{lemma}[{Cf.~\cite[IV, \S 1, Corollary 1.8]{Lang2002}}]Let $\F$ be a finite field with $q$ elements. Let $f$ be a
polynomial in $n$ variables over $\F$ such that the degree of $f$ in each variable
is less than $q$. If $f$ defines the zero function on $\F^{n}$, then $f = 0$.
\end{lemma}

\begin{corollary}\label{DegFLessZero}Let $\F$ be an arbitrary field. Let $f$ be a
polynomial in $n$ variables over $\F$ such that either $\F$ is infinite or the degree of $f$ in each variable
is less than $|\F|$. If $f$ defines the zero function on $\F^{n}$, then $f = 0$.
\end{corollary}

\begin{corollary}\label{DegFGLessEqual}Let $f, g$ be a
polynomials in $n$ variables over $\F$ such that either $\F$ is infinite or the degree of $f$ and $g$ in each variable
is less than $|\F|$. If $f$ and $g$ define the same function on $\F^{n}$, then $f = g$.
\end{corollary}

The notion of the radical of function is introduced following~\cite{Waterhouse1983}.

\begin{definition}[Cf.~{\cite[text at the beginning of Section~1]{Waterhouse1983}}]\label{def:radical}Let $\F$ be a field, $V$ be a finite-dimensional vector space over $\F$, and let $f$ be a function from $V$ to $\F$. The \emph{radical} of $f$, denoted by $\rad (f)$ is a subset of $V$ defined by
\[
\rad (f) = \{\mathbf w \mid f(\mathbf v + \lambda \mathbf w) = f(\mathbf v)\;\;\mbox{for all}\;\; \mathbf v \in V,\;\; \lambda \in \F\}.
\]
\end{definition}

\begin{definition}[Cf.~{\cite[text at the beginning of Section~1]{Waterhouse1983}}]\label{def:radicalfunc}If $\rad(f)$ is the radical of $f$ and $\pi_{\rad(f)}\colon V \to V / \rad(f)$ is the canonical projection, then $f_{\rad(f)}$ defined as a unique function from $V / \rad(f)$  to $\F$ with trivial radical such that $f = f_{\rad(f)}\circ \pi_{\rad(f)}$.
\end{definition}

The following technical facts involving $\rad(f)$ are used in the proof of Theorem~\ref{thm:homogenphipsi}.

\begin{lemma}\label{lem:StrangeCondImplyRadP}
Assume that $|\F| > 2$, $V$ is vector space over $\F$ and $f\colon V \to \F$ be a function. Let $\mathbf x \in V$. If  
\begin{equation}\label{lem:StrangeCondImplyRadP:eqq1}
f(\mathbf a + \lambda \mathbf x) = f(\mathbf a + \mathbf x)\quad\mbox{for all}\;\;\mathbf a \in V,\, 0 \neq \lambda \in \F, 
\end{equation}
then $\mathbf x \in \rad (f)$.
\end{lemma}
\begin{proof}Let $\mathbf x \in \F^n$ be such that the equality~\eqref{lem:StrangeCondImplyRadP:eqq1} holds for all $\mathbf a \in \F^n,\, 0 \neq \lambda \in \F$. We show that 
\begin{equation}\label{lem:StrangeCondImplyRadP:eq1}
f(\mathbf a + \mathbf x) = f(\mathbf a)\;\;\mbox{for all}\;\; \mathbf a \in V.
\end{equation}
The following two cases will be considered separately: (I) $\Char(\F) = 2$ and (II) $\Char(\F) \neq 2$.

\paragraph{Case (I): $\Char(\F) = 2$} Since $|\F| > 2,$ then there exists $\kappa \in \F$ such that $\kappa \neq 0$ and $\kappa \neq 1$. Then 
\begin{multline*}
f(\mathbf a + \mathbf x) = f(\mathbf a + (\kappa + (\kappa + 1))\mathbf x) = f(\mathbf a + \kappa\mathbf x + (\kappa + 1)\mathbf x)
= f(\left(\mathbf a + \kappa\mathbf x\right) + (\kappa + 1)\mathbf x)\\
\overset{\scriptsize \mbox{\eqref{lem:StrangeCondImplyRadP:eqq1} for}\;\mathbf a + \kappa\mathbf x \;\mbox{and}\; \kappa + 1 \neq 0}{=\joinrel=\joinrel=\joinrel=\joinrel=\joinrel=\joinrel=\joinrel=\joinrel=\joinrel=\joinrel=\joinrel=\joinrel=\joinrel=\joinrel=} f(\left(\mathbf a + \kappa\mathbf x\right) + \mathbf x) =  f(\left(\mathbf a + \mathbf x\right) + \kappa\mathbf x)\\
\overset{\scriptsize \mbox{\eqref{lem:StrangeCondImplyRadP:eqq1} for}\;\mathbf a + \mathbf x \;\mbox{and}\; \kappa \neq 0}{=\joinrel=\joinrel=\joinrel=\joinrel=\joinrel=\joinrel=\joinrel=\joinrel=\joinrel=\joinrel=\joinrel=\joinrel=} f(\left(\mathbf a + \mathbf x\right) + \mathbf x)
= f(\mathbf a + \left(\mathbf x + \mathbf x\right)) =  f(\mathbf a)
\end{multline*}
which establishes the equality~\eqref{lem:StrangeCondImplyRadP:eq1}.
\paragraph{Case (II): $\Char(\F) \neq 2$} Then $2 \neq 0$ and consequently
\begin{multline*}
f(\mathbf a + \mathbf x) \overset{\scriptsize \mbox{\eqref{lem:StrangeCondImplyRadP:eqq1} for}\;\mathbf a \;\mbox{and}\; 2 \neq 0}{=\joinrel=\joinrel=\joinrel=\joinrel=\joinrel=\joinrel=\joinrel=\joinrel=\joinrel=\joinrel=\joinrel=\joinrel=\joinrel=\joinrel=} f(\mathbf a + 2\mathbf x) = f(\left(\mathbf a + \mathbf x\right) + \mathbf x)\\
 \overset{\scriptsize \mbox{\eqref{lem:StrangeCondImplyRadP:eqq1} for}\;\mathbf a + \mathbf x \;\mbox{and}\; -1 \neq 0}{=\joinrel=\joinrel=\joinrel=\joinrel=\joinrel=\joinrel=\joinrel=\joinrel=\joinrel=\joinrel=\joinrel=\joinrel=\joinrel=\joinrel=} f(\left(\mathbf a + \mathbf x\right) + (-1)\mathbf x) = f(\mathbf a).
\end{multline*}
This establishes the equality~\eqref{lem:StrangeCondImplyRadP:eq1}.\bigskip

Returning to the statement of the lemma, we show that $\mathbf x \in \rad(f)$ by the definition of $\rad(f)$. Indeed, the equality
\begin{equation}\label{lem:StrangeCondImplyRadP:eq2}
f(\mathbf a + \lambda\mathbf x) = f(\mathbf a)
\end{equation}
clearly holds for all $\mathbf a \in V$ and $\lambda = 0$. If $\lambda \neq 0$, then~\eqref{lem:StrangeCondImplyRadP:eq2} follows from equalities~\eqref{lem:StrangeCondImplyRadP:eqq1} and~\eqref{lem:StrangeCondImplyRadP:eq1} aligned together.

Thus, the equality~\eqref{lem:StrangeCondImplyRadP:eq2} holds for all $\mathbf a \in V$ and $\lambda \in \F$. Therefore, $\mathbf x \in \rad(f)$.
\end{proof}

\begin{remark}
 If $|\F| = 2$, then the equality~\eqref{lem:StrangeCondImplyRadP:eqq1} holds for all $\mathbf x \in V$ and the statement of Lemma~\ref{lem:StrangeCondImplyRadP} does not hold. Indeed, it does not hold in particular for $f = \id \colon \F \to \F$ and $\mathbf x = 1 \in \F$.
\end{remark}

\begin{lemma}\label{lem:TwoCondLift}
Let $V$ be a vector space over $\F$ and $f\colon V \to \F$ be a function, $\phi, \psi \colon V \to V$ be maps such that
\begin{equation}\label{lem:TwoCondLift:eqq1}
f(\mathbf x + \lambda \mathbf y) = f(\phi(\mathbf x) + \lambda (\mathbf y))\quad\mbox{for all}\;\;\mathbf x, \mathbf y \in V, \lambda \in \F,
\end{equation}
$\psi_{\rad(f)}\colon V/\rad(f) \to V/\rad(f)$ be a map satisfying
\begin{equation}\label{lem:TwoCondLift:eqq2}
\psi_{\rad(f)} \circ \pi_{\rad(f)} = \pi_{\rad(f)} \circ \psi. 
\end{equation}
Then
\begin{equation*}\label{lem:TwoCondLift:eqq3}
f_{\rad(f)}(\pi_{\rad(f)}(\mathbf x) + \lambda \mathbf y) = f_{\rad(f)}(\pi_{\rad(f)}(\phi(\mathbf x)) + \lambda \psi_{\rad(f)}(\mathbf y))
\end{equation*}
for all $\mathbf x \in V, \mathbf y \in V/\rad(f), \lambda \in \F.$
\end{lemma}
\begin{proof}Let $\mathbf x \in V, \mathbf y \in V/\rad(f), \lambda \in \F$ and $\mathbf z \in V$ be such that $\pi_{\rad(f)}(\mathbf z) = \mathbf y.$
Then
\begin{multline*}
f_{\rad(f)}(\pi_{\rad(f)}(\mathbf x) + \lambda \mathbf y) = f_{\rad(f)}(\pi_{\rad(f)}(\mathbf x) + \lambda \pi_{\rad(f)}(\mathbf z))\\
 \overset{\scriptsize\mbox{$\pi_{\rad(f)}$ is linear}}{=\joinrel=\joinrel=\joinrel=\joinrel=\joinrel=\joinrel=\joinrel=\joinrel=\joinrel=\joinrel=} f_{\rad(f)}(\pi_{\rad(f)}(\mathbf x + \lambda \mathbf z))
 = f_{\rad(f)}\circ \pi_{\rad(f)}(\mathbf x + \lambda \mathbf z)\\
\overset{\scriptsize\mbox{Definition of $f_{\rad(f)}$}}{=\joinrel=\joinrel=\joinrel=\joinrel=\joinrel=\joinrel=\joinrel=\joinrel=\joinrel=\joinrel=}
f(\mathbf x + \lambda \mathbf z)
\overset{\scriptsize\mbox{\eqref{lem:TwoCondLift:eqq1}}}{=\joinrel=\joinrel=\joinrel=}
f(\phi(\mathbf x) + \lambda \psi(\mathbf z))\\
\overset{\scriptsize\mbox{Definition of $f_{\rad(f)}$}}{=\joinrel=\joinrel=\joinrel=\joinrel=\joinrel=\joinrel=\joinrel=\joinrel=\joinrel=\joinrel=}
f_{\rad(f)}\circ \pi_{\rad(f)}(\phi(\mathbf x) + \lambda \psi(\mathbf z))\phantom{XXXXXXXXXXXX}\\
\overset{\scriptsize\mbox{$\pi_{\rad(f)}$ is linear}}{=\joinrel=\joinrel=\joinrel=\joinrel=\joinrel=\joinrel=\joinrel=\joinrel=\joinrel=\joinrel=}
f_{\rad(f)}\left(\pi_{\rad(f)} (\phi(\mathbf x)) + \lambda \pi_{\rad(f)}\circ \psi(\mathbf z))\right)\\
\overset{\scriptsize\mbox{\eqref{lem:TwoCondLift:eqq2}}}{=\joinrel=\joinrel=\joinrel=\joinrel=}
f_{\rad(f)}\left(\pi_{\rad(f)} (\phi(\mathbf x)) + \lambda \psi_{\rad(f)} \circ  \pi_{\rad(f)}(\mathbf z))\right)\\
\overset{\scriptsize\mbox{Definition of $\mathbf z$}}{=\joinrel=\joinrel=\joinrel=\joinrel=\joinrel=\joinrel=\joinrel=\joinrel=\joinrel=\joinrel=}
 f_{\rad(f)}(\pi_{\rad(f)}(\phi(\mathbf x)) + \lambda \psi_{\rad(f)}(\mathbf y))
\end{multline*} 
Thus, the equality~\eqref{lem:TwoCondLift:eqq3} holds for all $\mathbf x \in V, \mathbf y \in V/\rad(f), \lambda \in \F$.
\end{proof}

\begin{definition}[{Cf.~\cite[IV,\S 1, the definition on p.~178]{Lang2002}}]\label{def:onedimder}If $f(X) =
a_n x^n + \ldots + a_0 \in \F[x]$, then \textbf{derivative} of $f$ is a polynomial belonging to $\F[x]$ which is denoted by $f'$ and is defined by
\[
f'(X)= na_n x^{n-1} + \ldots + a_1.
\]
\end{definition}

\begin{lemma}[{Cf.~\cite[IV, \S 1, Proposition 1.12]{Lang2002}}]\label{lem:DegGE1DerNonZero}Let $f \in \F[x]$. If $\Char(\F) = 0$ and $\deg(f) \ge 1$, then $f' \neq 0$.
\end{lemma}

\begin{corollary}\label{cor:DerZeroPolConst}Let $f \in \F[x]$. If $\Char(\F) = 0$ and $f' = 0$, then $f = \mathrm{const}$.
\end{corollary}
%
%

\begin{definition}
Let $P \in \F[x_1,\ldots, x_n]$. Then $\frac{\partial P}{\partial \mathbf v}(\mathbf a) \in \F[x_1,\ldots, x_n; y_1, \ldots, y_n]$ denotes the \emph{directional derivative of $P$ along $\mathbf v$ at $\mathbf a$} and is defined as a coefficient at $t$ in the following equality in $\F[x_1,\ldots, x_n; y_1, \ldots, y_n; t]$
\[
P(\mathbf a + t \mathbf v) = P(\mathbf a)  + P_1(\mathbf a, \mathbf v)t + \mbox{(terms of higher degree in $t$)},
\]
where the coordinates of $\mathbf a$ correspond to $x_1,\ldots, x_n$, and  the coordinates of $\mathbf v$ correspond to $y_1,\ldots, y_n$. 
\end{definition}

\begin{lemma}\label{lem:DirDerLin}
If $P \in \F[x_1,\ldots, x_n]$ and $\mathbf a \in \F^n$ are fixed, then the function $\mathbf v \mapsto  \frac{\partial P}{\partial \mathbf v}(\mathbf a)$ is linear.
\end{lemma}
\begin{proof}Let us show that 
\begin{equation}\label{lem:DirDerLin:eq101}
\frac{\partial P}{\partial (\lambda \mathbf v)}(\mathbf a) = \lambda\frac{\partial P}{\partial \mathbf v}(\mathbf a)\;\;\mbox{for all}\;\;\mathbf a, \mathbf v \in \F^n\;\;\mbox{and}\;\;\lambda \in \F.
\end{equation}
Let $\mathbf a, \mathbf v \in \F^n$ and $\lambda \in \F$. 
On the one hand, by the definition of $\frac{\partial P}{\partial \mathbf v}(\mathbf a)$,
\begin{equation}\label{lem:DirDerLin:eq1}
P(\mathbf a + t (\lambda \mathbf v)) = P(\mathbf a + (\lambda t)  \mathbf v)\\ 
= P(\mathbf a)  + \lambda \frac{\partial P}{\partial \mathbf v}(\mathbf a)t + \mbox{(terms of higher degree in $t$)}.
\end{equation}
On the other hand, by the definition of $\frac{\partial P}{\partial \lambda (\mathbf v)}(\mathbf a)$,
\begin{equation}\label{lem:DirDerLin:eq2}
P(\mathbf a + t (\lambda \mathbf v)) 
= P(\mathbf a)  + \frac{\partial P}{\partial (\lambda \mathbf v)}(\mathbf a)t + \mbox{(terms of higher degree in $t$)}.
\end{equation}
By comparing the coefficients of $t$ in~\eqref{lem:DirDerLin:eq1} and~\eqref{lem:DirDerLin:eq2} we obtain the equality~\eqref{lem:DirDerLin:eq101}. 

Now let us show that 
\begin{equation}\label{lem:DirDerLin:eq102}
\frac{\partial P}{\partial (\mathbf u + \mathbf v)}(\mathbf a) = \frac{\partial P}{\partial \mathbf u}(\mathbf a) + \frac{\partial P}{\partial \mathbf v}(\mathbf a)\;\;\mbox{for all}\;\;\mathbf a, \mathbf u, \mathbf v \in \F^n.
\end{equation}
Let $\mathbf a, \mathbf u, \mathbf v \in \F^n$. On the one hand, by the definition of $\frac{\partial P}{\partial \mathbf u}(\mathbf a)$ and $\frac{\partial P}{\partial \mathbf v}(\mathbf a)$,
\begin{multline}\label{lem:DirDerLin:eq3}
P(\mathbf a + t (\mathbf u + \mathbf v)) = P((\mathbf a + t\mathbf u) +   t\mathbf v)\\
= P(\mathbf a + t\mathbf u)  +  \frac{\partial P}{\partial \mathbf v}(\mathbf a + t\mathbf u)t\phantom{XXXXXXXXXXX}\\\phantom{XXXXXXXXXXX} + \mbox{(terms of higher degree in $t$)}\\
= P(\mathbf a + t\mathbf u)  +  \left(\frac{\partial P}{\partial \mathbf v}(\mathbf a)+\mbox{(terms of higher degree in $t$)}\right)t\phantom{XXXXXXX}\\\phantom{XXXXXXXXXXX} + \mbox{(terms of higher degree in $t$)}\\
=  P(\mathbf a)  + \frac{\partial P}{\partial \mathbf u}(\mathbf a) t + \mbox{(terms of higher degree in $t$)} + \frac{\partial P}{\partial \mathbf v}(\mathbf a) t\phantom{XXXXXXX}\\\phantom{XXXXXXXXXXX} + \mbox{(terms of higher degree in $t$)}\\
=  P(\mathbf a) + \left(\frac{\partial P}{\partial \mathbf u}(\mathbf a) + \frac{\partial P}{\partial \mathbf v}(\mathbf a) \right)t + \mbox{(terms of higher degree in $t$)},
\end{multline}

On the other hand, by the definition of $\frac{\partial P}{\partial (\mathbf u + \mathbf v)}(\mathbf a)$,
\begin{equation}\label{lem:DirDerLin:eq4}
P(\mathbf a +  t (\mathbf u + \mathbf v)) 
= P(\mathbf a)  + \frac{\partial P}{\partial (\mathbf u + \mathbf v)}(\mathbf a)t + \mbox{(terms of higher degree in $t$)}.
\end{equation}
%
By comparing the coefficients of $t$ in~\eqref{lem:DirDerLin:eq3} and~\eqref{lem:DirDerLin:eq4} we obtain the equality~\eqref{lem:DirDerLin:eq102}.
%
\end{proof}

\begin{definition}\label{def:PAV}
Let $P \in \F[x_1,\ldots, x_n]$ and $\mathbf a, \mathbf v \in \F^n$ are fixed. Then by $P_{\mathbf a, \mathbf v} \in \F[t]$ we denote a polynomial defined by
\begin{equation}\label{eq:def:PAV}
P_{\mathbf a, \mathbf v} = P(\mathbf a + t \mathbf v).
\end{equation}

\end{definition}

\begin{lemma}\label{def:partDer}Let $P \in \F[x_1,\ldots, x_n]$ and $\mathbf a, \mathbf v \in \F^n$. Then $\deg(P_{\mathbf a, \mathbf v}) \le \deg(P)$ and 
\begin{equation}\label{def:partDer:eqq1}
\frac{\partial P}{\partial \mathbf v}(\mathbf a) = P'_{\mathbf a, \mathbf v}(0)\;\;\mbox{for all}\;\; \mathbf v \in \F^n.
\end{equation}
\end{lemma}
\begin{proof}Indeed, the inequality $\deg(P_{\mathbf a, \mathbf v}) \le \deg(P)$ follows directly from the definition of degree of multivariate polynomial and the definition of $P_{\mathbf a, \mathbf v}$. 

In addition, we have an equality by the definition of $\frac{\partial P}{\partial \mathbf v}(\mathbf a)$ and $P_{\mathbf a, \mathbf v}$,
\begin{multline*}
P_{\mathbf a, \mathbf v}(t) \overset{\mbox{\scriptsize Definition of $P_{\mathbf a, \mathbf v}$}}{=\joinrel=\joinrel=\joinrel=\joinrel=\joinrel=\joinrel=\joinrel=\joinrel=} P(\mathbf a + t \mathbf v) 
\overset{\mbox{\scriptsize Definition of $\frac{\partial P}{\partial \mathbf v}(\mathbf a)$}}{=\joinrel=\joinrel=\joinrel=\joinrel=\joinrel=\joinrel=\joinrel=\joinrel=} P(\mathbf a)  + \frac{\partial P}{\partial \mathbf v}(\mathbf a)t\phantom{XXXX}\\\phantom{XXXXXXXXXXX}  + \mbox{(terms of higher degree in $t$)}.
\end{multline*}
Therefore, the equality~\eqref{def:partDer:eqq1} holds by Definition~\ref{def:onedimder}.
\end{proof}


\section{Definition and properties of $\mathcal L_P$}
\label{sec:linspace}

\begin{definition} Let $P \in \F[x_1,\ldots, x_n]$ and $\mathbf a, \mathbf v \in \F^n$. By $l_{P,\mathbf a} \colon \mathbb F^n \to \mathbb F$ we denote a linear function defined by
\[
l_{P, \mathbf a}(\mathbf v) = \frac{\partial P}{\partial \mathbf v}(\mathbf a).
\] 
\end{definition}

\begin{definition}\label{def:LP}
Let $P \in \F[x_1,\ldots, x_n]$. By $\mathcal L_P \subseteq (\F^n)^*$ we denote a vector space spanned by all the linear functions of the form $l_{P, \mathbf a}$ for some $\mathbf a \in \F^n.$ That is,
\[
\mathcal L_P = \Span (\{\mathbf v \mapsto \frac{\partial P}{\partial \mathbf v}(\mathbf a) \mid \mathbf a \in \mathbb F^n\}).
\]
Equivalently, $\mathcal L_{P} \subseteq {\F^n}^*$ spanned by the range of the gradient field of $P \in \F[x_1,\ldots, x_n]$.
\end{definition}

\begin{remark}Although in certain cases the set $S = \{\mathbf v \mapsto \frac{\partial P}{\partial \mathbf v}(\mathbf a) \mid \mathbf a \in \mathbb F^n\}$ is a vector space, in the general case it is not true. For example, if $\F = \mathbb R, P = x^3,$ then $S = \{\mathbf v \mapsto \alpha \mathbf v \mid \alpha \ge 0\},$ which is not a vector space.
\end{remark}

\begin{lemma}\label{lem:RadAnnLP}Let $P \in \F[x_1,\ldots, x_n]$. If $\mathbf v \in \mathcal \rad(P)$ and $\deg(P) < |\F|$, then $l(\mathbf v) = 0$ for all $l \in \mathcal L_P.$
\end{lemma}
\begin{proof}Let $\mathbf a \in \F^n$ and $\mathbf v \in \rad(P)$. Then 
\begin{equation}\label{lem:RadAnnLP:eq1}
P_{\mathbf a, \mathbf v}(t) = 0\;\;\mbox{for all}\;\;t \in \F
\end{equation}
by the definition of  $P_{\mathbf a, \mathbf v} \in \F[t]$ (Definition~\ref{def:PAV}). Since $\deg(P_{\mathbf a, \mathbf v}) \leq \deg(P)$ by Lemma~\ref{def:partDer} and $\deg(P) < |\F|$ by the statement of the lemma, then\linebreak $\deg(P_{\mathbf a, \mathbf v}) < |\F|$. Hence,
\begin{equation}\label{lem:RadAnnLP:eq2}
P_{\mathbf a, \mathbf v} = 0\;\;\mbox{as an element of}\;\;\F[t]
\end{equation}
by Lemma~\ref{DegFLessZero}. Using the definition of $P'_{\mathbf a, \mathbf v}$ (Definition~\ref{def:onedimder}), we conclude from~\eqref{lem:RadAnnLP:eq2} that 
\begin{equation}\label{lem:RadAnnLP:eq4}
P'_{\mathbf a, \mathbf v}(0) = 0.
\end{equation}
Therefore,
\begin{equation}\label{lem:RadAnnLP:eq3}
\frac{\partial P}{\partial \mathbf v}(\mathbf a) \overset{\mbox{\scriptsize Lemma~\ref{def:partDer}}}{=\joinrel=\joinrel=\joinrel=\joinrel=\joinrel=\joinrel=\joinrel=} P'_{\mathbf a, \mathbf v}(0)  \overset{\eqref{lem:RadAnnLP:eq4}}{=\joinrel=\joinrel=} 0.
\end{equation}

Thus, $l_{P, \mathbf a}(\mathbf v) = 0$ for all $\mathbf a$. Since every $l \in \mathcal L_P$ is a linear combination of $l_{P, \mathbf a_1}, \ldots, l_{P, \mathbf a_k}$ for some $k \ge 0$ and $\mathbf a_1,\ldots, \mathbf a_k \in \F^n$, then we conclude that $l(\mathbf v) = 0$ for all $l \in \mathcal L_P.$
\end{proof}

Now we discuss when the converse if also true.

\begin{lemma}\label{lem:charzerokp} Assume that $\Char (\F) = 0$. Let $P \in \F[x_1,\ldots, x_n]$ and $\mathbf v \in \F^n.$ If $l(\mathbf v) = 0$ for all $l \in \mathcal L_P$, then $\mathbf v \in \mathcal \rad(P).$
\end{lemma}
\begin{proof}
Let $\mathbf a \in \F^n$, $\kappa \in \F$ and $\mathbf v \in \rad(P)$. Since $\Char (\F) = 0$, then $\F$ is an infinite field. Hence, we have the following equality in $\F[t]$
\begin{equation}\label{lem:charzerokp:eq1}
P_{\mathbf a + \kappa \mathbf v, \mathbf v} = P(\mathbf a + \kappa \mathbf v + t \mathbf v) = P_{\mathbf a, \mathbf v}(t + \kappa),
\end{equation}
where $P_{\mathbf a + \kappa \mathbf v, \mathbf v}$ and $P_{\mathbf a, \mathbf v}$ defined in Definition~\ref{def:PAV}. Since $l_{P, \mathbf a} \in \mathcal L_P$ for all $\mathbf a \in \F^n$,
\begin{multline}\label{lem:charzerokp:eq2}
P_{\mathbf a, \mathbf v}'(\kappa) = P_{\mathbf a, \mathbf v}(t + \kappa)'(0) \overset{\eqref{lem:charzerokp:eq1}}{=\joinrel=\joinrel=} P_{\mathbf a + \kappa \mathbf v, \mathbf v}'(0) \\ \overset{\mbox{\scriptsize Definition~\ref{def:PAV}}}{=\joinrel=\joinrel=\joinrel=\joinrel=\joinrel=\joinrel=\joinrel=\joinrel=}  l_{P, \mathbf a + \kappa \mathbf v}(\mathbf v) \overset{\mbox{\scriptsize the conditions of the lemma}}{=\joinrel=\joinrel=\joinrel=\joinrel=\joinrel=\joinrel=\joinrel=\joinrel=\joinrel=\joinrel=\joinrel=\joinrel=\joinrel=\joinrel=\joinrel=} 0\;\;\mbox{for all}\;\;\mathbf a \in \F^n.
\end{multline}
Since the equality~\eqref{lem:charzerokp:eq2} holds for all $\kappa \in \F$ and $\Char(\F) = 0$ by the statement of the lemma, then \[
P_{\mathbf a, \mathbf v} =  \mathrm{const}
\]
by Lemma~\ref{cor:DerZeroPolConst}. By the definition of $P_{\mathbf a, \mathbf v}$ this means that 
\begin{equation}\label{lem:charzerokp:eq3}
P_{\mathbf a, \mathbf v}(\mathbf a + \kappa \mathbf v) = P_{\mathbf a, \mathbf v}(\mathbf a)\;\;\mbox{for all}\;\;\kappa \in \F.
\end{equation}
The equality~\eqref{lem:charzerokp:eq3} holds for all $\mathbf a \in \F^n$. Therefore, $\mathbf v \in \mathcal \rad(P).$
\end{proof}

\begin{remark}In the case $\Char (\F) \neq 0$ we need additional assumption on $P$ for the statement of Lemma~\ref{lem:charzerokp} to hold.  For example, if $\Char (\F) = p,$ $P = x^p \in \F[x],$ then $\frac{\partial P}{\partial \mathbf{v}}(\mathbf a) = 0$ for all $\mathbf a\in \F,$ and also $(a + \lambda b)^p = a^p$ for nonzero $\lambda$ implies $b = 0$ and hence $\rad(P) = \{0\}.$ One of the possible necessary and sufficient condition is formulated in Lemma~\ref{lem:dzeroeqk}\ref{lem:dzeroeqk:part2} below.
\end{remark}

\begin{lemma}\label{lem:dzeroeqk}Let $P \in \F[x_1,\ldots, x_n]$ be such that $\deg(P) < |\F|$ . Then
\begin{enumerate}[label=(\alph*), ref=(\alph*)]
\item\label{lem:dzeroeqk:part1} $\dim (\mathcal L_P) + \dim (\rad(P)) \leqslant n$ for all $P \in \F[x_1,\ldots, x_n]$.
\item\label{lem:dzeroeqk:part2} $\{\mathbf v \in \F^n \mid l(\mathbf v) = 0\;\mbox{for all} \;l \in \mathcal L_P\} = \rad(P)$ if and only if $\dim (\rad(P)) + \dim (\mathcal L_P) = n.$
\end{enumerate}
\end{lemma}
\begin{proof}Let $\pi_{\rad(P)} \colon \F^n \to \F^n/\rad(P)$ be a canonic projection. Since $\pi_{\rad(P)}$ is surjection, then $\pi_{\rad(P)}^*$ is injection and consequently
\begin{equation}\label{lem:dzeroeqk:eq1}
\dim(\Img (\pi_{\rad(P)}^*)) = \dim((\F^n/\rad(P))^*)
\end{equation}

In addition, Lemma~\ref{lem:RadAnnLP} and the universal property of $\pi_{\rad(P)}$ imply that 
\begin{equation}\label{lem:dzeroeqk:eq6}
\mathcal L_P \subseteq \Img (\pi_{\rad(P)}^*)
\end{equation}

Now let us proof the part~\ref{lem:dzeroeqk:part1} of the lemma. We have
 \begin{equation}\label{lem:dzeroeqk:eq2}
\dim (\mathcal L_P) \overset{\eqref{lem:dzeroeqk:eq6}}{\leq} \Img (\pi_{\rad(P)}^*) \stackrel{\eqref{lem:dzeroeqk:eq1}}{=\joinrel=\joinrel=} \dim((\F^n/\rad(P))^*) = n - \dim(\rad(P)).
\end{equation}
Therefore,
\begin{equation*}
\dim (\mathcal L_P) +  \dim(\rad(P)) \le n.
\end{equation*}
This establishes the part~\ref{lem:dzeroeqk:part1} of the lemma.

Let us prove the part~\ref{lem:dzeroeqk:part2} of the lemma. Let $(\mathcal L_P)^0 \subseteq \F^n$ be defined by 
\begin{equation}\label{lem:dzeroeqk:eq11}
(\mathcal L_P)^0 = \{\mathbf v \in \F^n \mid l(\mathbf v) = 0\;\mbox{for all} \;l \in \mathcal L_P\}. 
\end{equation}
From the basic properties of the vector spaces and their annihilators we have
\begin{equation}\label{lem:dzeroeqk:eq4} 
\dim(\mathcal L_P) + \dim ((\mathcal L_P)^0) = n.
\end{equation}

Now let us prove sufficiency. Assume that 
\begin{equation}\label{lem:dzeroeqk:eq3}
\{\mathbf v \in \F^n \mid l(\mathbf v) = 0\;\mbox{for all} \;l \in \mathcal L_P\} = \rad(P). 
\end{equation}
By aligning~\eqref{lem:dzeroeqk:eq11} and~\eqref{lem:dzeroeqk:eq3} together and substituting the result into~\eqref{lem:dzeroeqk:eq4} we conclude that 
\begin{equation}\label{lem:dzeroeqk:eq5} 
\dim(\mathcal L_P) + \dim (\rad(P)) = n
\end{equation}
 which establishes the sufficiency.

Let us prove the necessity. Assume that the equality~\eqref{lem:dzeroeqk:eq5} holds. Hence, 
\begin{equation}\label{lem:dzeroeqk:eq12} 
\dim(\rad(P)) = \dim(\mathcal L_P^0)
\end{equation}
by~\eqref{lem:dzeroeqk:eq4} and~\eqref{lem:dzeroeqk:eq5}. From Lemma~\ref{lem:RadAnnLP} we have an inclusion
\begin{equation}\label{lem:dzeroeqk:eq13}
\rad(P) \subseteq (\mathcal L_P)^0.
\end{equation}
From~\eqref{lem:dzeroeqk:eq12} and~\eqref{lem:dzeroeqk:eq13} we conclude that 
$$\rad(P) = (\mathcal L_P)^0 \overset{\eqref{lem:dzeroeqk:eq11}}{=\joinrel=\joinrel=\joinrel} \{\mathbf v \in \F^n \mid l(\mathbf v) = 0\;\mbox{for all} \;l \in \mathcal L_P\}$$
which establishes the necessity.
\end{proof}

\begin{remark}\label{rem:LPUseful}As the reader can see in the proofs of Lemma~\ref{lem:NKEvenAdditionPreserverIsZero} and\linebreak Lemma~\ref{lem:NKOddRadDetNK}, Lemma~\ref{lem:dzeroeqk} is also used for finding the radical of the polynomial $P \in \F[x_1,\ldots, x_n]$. 
\end{remark}

\begin{corollary}\label{cor:charzerokp}If $\Char (\F) = 0$, then $\dim (\rad(P)) + \dim (\mathcal L_P) = n$ for all $P \in \F[x_1,\ldots, x_n]$.
\end{corollary}
\begin{proof}
Let $P \in \F[x_1,\ldots, x_n]$. Lemma~\ref{lem:charzerokp} implies that 
$$
\{\mathbf v \in \F^n \mid l(\mathbf v) = 0\;\mbox{for all} \;l \in \mathcal L_P\} \subseteq \rad(P).
$$
In addition,
$$
\{\mathbf v \in \F^n \mid l(\mathbf v) = 0\;\mbox{for all} \;l \in \mathcal L_P\} \supseteq \rad(P)
$$
by Lemma~\ref{lem:RadAnnLP}. Therefore,
$$
\{\mathbf v \in \F^n \mid l(\mathbf v) = 0\;\mbox{for all} \;l \in \mathcal L_P\} = \rad(P).
$$
Hence, $\dim (\rad(P)) + \dim (\mathcal L_P) = n$ by Lemma~\ref{lem:dzeroeqk}\ref{lem:dzeroeqk:part2}.
\end{proof}

In the lemma below we show that the main condition appeared in this paper (the condition~\eqref{thm:homogenphipsiintro:eqq1}) implies that the values of elements of $\mathcal L_P$ at certain points are equal. This lemma is used in the proof of Lemma~\ref{lem:PPhiPsi} in the next section.

\begin{lemma}\label{lem:eqpd}Let $P \in \F[x_1,\ldots, x_n]$ and  $\mathbf x_1, \mathbf x_2, \mathbf y_1, \mathbf y_2 \in \F^n$. If $\deg(P) < |\F|$ and
\begin{equation}\label{lem:eqpd:eqlem1}
P(\mathbf x_1 + \lambda \mathbf y_1) = P(\mathbf x_2 + \lambda \mathbf y_2)\;\;\mbox{for all}\;\;\lambda \in \F, 
\end{equation}
then
\begin{equation*}\label{lem:eqpd:eqlem2}
l_{P,\mathbf x_1}(\mathbf y_1) = l_{P, \mathbf x_2}(\mathbf y_2).
\end{equation*}
\end{lemma}
\begin{proof}Indeed, the equality~\eqref{lem:eqpd:eqlem1} implies that 
\begin{equation}\label{lem:eqpd:eq1}
P_{\mathbf x_1, \mathbf y_1}(\lambda) = P_{\mathbf x_2, \mathbf y_2}(\lambda)\;\;\mbox{for all}\;\;\lambda \in \F,
\end{equation}
where $P_{\mathbf x_1, \mathbf y_1}, P_{\mathbf x_2, \mathbf y_2} \in \F[t]$ are defined in Definition~\ref{def:PAV}. Since $\deg(P_{\mathbf x_1, \mathbf y_1}) \leq \deg(P)$ and $\deg(P_{\mathbf x_2, \mathbf y_2}) \leq \deg(P)$ by Lemma~\ref{def:partDer}, and $\deg(P) < |\F|$ by the statement of the lemma, then $\deg(P_{\mathbf x_1, \mathbf y_1}) < |\F|$ and $\deg(P_{\mathbf x_2, \mathbf y_2}) < |\F|$. Therefore, by~\eqref{lem:eqpd:eq1} and Lemma~\ref{DegFGLessEqual} we conclude that
\begin{equation}\label{lem:eqpd:eq2}
P_{\mathbf x_1, \mathbf y_1} = P_{\mathbf x_2, \mathbf y_2}\;\;\mbox{as elements of}\;\F[t].
\end{equation}
Hence, $P'_{\mathbf x_1, \mathbf y_1} = P'_{\mathbf x_2, \mathbf y_2}$ and, in particular,
$P'_{\mathbf x_1, \mathbf y_1}(0) = P'_{\mathbf x_2, \mathbf y_2}(0)$

Therefore,
\[
l_{P,\mathbf x_1}(\mathbf y_1) = P'_{\mathbf x_1, \mathbf y_1}(0) = P'_{\mathbf x_2, \mathbf y_2}(0) =  l_{P, \mathbf x_2}(\mathbf y_2)
\]
by Lemma~\ref{def:partDer} and the definitions of $l_{P,\mathbf x_1}$ and $l_{P,\mathbf x_2}$.
\end{proof}

\section{Pairs of maps $\phi, \psi$ satisfying the condition $P(\mathbf{x} + \lambda \mathbf{y}) = P(\phi(\mathbf{x}) + \lambda \psi(\mathbf{y}))$ for all $\mathbf{x},\mathbf{y} \in \mathbb F^n$ and $\lambda \in \F$}
\label{sec:factspres}

%
%
In this section we prove the main theorem of this article (Theorem~\ref{thm:homogenphipsi}). This theorem relies on Lemma~\ref{lem:PPhiPsi} which holds without an assumption that the polynomial $P \in \F[x_1,\ldots, x_n]$ is homogeneous. 

The following simple observation preceds the formulation of Lemma~\ref{lem:PPhiPsi} and is used in its  proof.

\begin{lemma}\label{lem:condimpliesphi}Let $P \in \F[x_1,\ldots, x_n]$, and let $\phi, \psi \colon \F^n \to \F^n$ be two maps such that
\begin{equation}\label{lem:condimpliesphi:eqq1}
P(\mathbf{x} + \lambda \mathbf{y}) = P(\phi(\mathbf{x}) + \lambda \psi(\mathbf{y}))\;\;\mbox{for all}\;\;\mathbf{x},\mathbf{y} \in \mathbb \F^n,\; \lambda \in \F.
\end{equation}
Then $P(\phi(\mathbf x)) = P (\mathbf x)$ for all $\mathbf x \in \F^n.$
\end{lemma}
\begin{proof}Consider~\eqref{lem:condimpliesphi:eqq1} for $\lambda = 0.$
\end{proof}

\begin{lemma}\label{lem:PPhiPsi}Let $P \in \F[x_1,\ldots, x_n]$ be a polynomial such that $\deg (P) < |\F|$ and $\dim (\mathcal L_P) + \dim (\rad(P)) = n$.  If  $\phi, \psi \colon \F^n \to \F^n$ are two maps such that
\begin{equation}\label{lem:PPhiPsi:eqq1}
P(\mathbf{x} + \lambda \mathbf{y}) = P(\phi(\mathbf{x}) + \lambda \psi(\mathbf{y}))\;\;\mbox{for all}\;\;\mathbf{x},\mathbf{y} \in \mathbb \F^n,\; \lambda \in \F,
\end{equation}
then there exists a unique bijective linear map $\psi_{\rad(P)}$ on $\F^n / \rad(P)$ such that 
\begin{equation}\label{lem:PPhiPsi:eq}
\psi_{\rad(P)} \circ \pi_{\rad(P)} = \pi_{\rad(P)} \circ \psi,
\end{equation}
 where $\pi_{\rad(P)}$ is a canonical projection on $\F^n / \rad(P)$.
\end{lemma} 

\begin{proof} 
Let $\phi, \psi \colon \F^n \to \F^n$ be two maps satisfying the condition~\eqref{lem:PPhiPsi:eqq1}. Let $d = \dim (\mathcal L_P)$ and $\mathbf a_1, \ldots \mathbf a_d \in \F$ be such that $\beta \colon l_{P, \mathbf a_1}, \ldots, l_{P, \mathbf a_d}$ forms a basis of $\L_P$. Let $\mathbf e_1, \ldots, \mathbf e_d \in \F^n$ be linearly independent vectors such that 
\begin{equation}\label{lem:PPhiPsi:eq6}
l_{P, \mathbf a_i}(\mathbf e_j) = \delta_{i\,j}\;\;\mbox{for all}\;\; 1 \le i, j \le d
\end{equation}
 (such vectors could be obtained by extending of $\beta$ to a basis of $\left(\F^n\right)^*$ and finding dual of this basis). Let $\tilde{\psi}\colon \F^n \to \F^n$ be defined by
\begin{equation}\label{lem:PPhiPsi:eq4}
\tilde{\psi}(\mathbf x) = l_{P, \mathbf a_1} \left(\psi(\mathbf x)\right)\mathbf e_1 +  \ldots + l_{P, \mathbf a_d} \left(\psi(\mathbf x)\right)\mathbf e_d\quad\mbox{for all}\;\; \mathbf x \in \F^n.
\end{equation}

The following four claims will be proven sequentially:
\begin{enumerate}[label=(\roman*), ref=(\roman*)]
\item\label{lem:PPhiPsi:part1} The map $\tilde{\psi}$ is linear and 
\begin{equation}\label{lem:PPhiPsi:eqq4}
\tilde{\psi}(\rad(P)) = \{\mathbf 0\}.
\end{equation}
\item\label{lem:PPhiPsi:part2} 
\begin{equation}\label{lem:PPhiPsi:eqq6}
\pi_{\rad} \circ \tilde{\psi} = \pi_{\rad} \circ \psi.
\end{equation}
\item\label{lem:PPhiPsi:part3} There exists a unique linear map $\psi_{\rad(P)} \colon \F^n / \rad(P) \to \F^n / \rad(P)$ such that 
\begin{equation}\label{lem:PPhiPsi:eqq5}
\psi_{\rad(P)} \circ \pi_{\rad(P)} = \pi_{\rad(P)} \circ \tilde{\psi}.
\end{equation}
\item\label{lem:PPhiPsi:part4} Every linear map $\psi_{\rad(P)} \colon \F^n / \rad(P) \to \F^n / \rad(P)$ satisfying the equation~\eqref{lem:PPhiPsi:eq} is bijective.
\end{enumerate}

The claims~\ref{lem:PPhiPsi:part2}, \ref{lem:PPhiPsi:part3} and~\ref{lem:PPhiPsi:part4} together clearly imply the statement of the lemma.

\paragraph{Claim~\ref{lem:PPhiPsi:part1}}
The proof of this claim employs a certain form of the beautiful argument of Tan and Wang (see Proposition 2.1 in~\cite{TAN2003311}).

Since $\phi$ and  $\psi$ satisfy condition~\eqref{lem:PPhiPsi:eqq1}, by Lemma~\ref{lem:eqpd} we have
\begin{equation}\label{lem:PPhiPsi:eq1}
l_{P,\mathbf a_i}(\mathbf x) = l_{P, \phi(\mathbf a_i)}(\psi(\mathbf x))\;\;\mbox{for all}\;\;1 \le i \le d\;\;\mbox{and}\;\;\mathbf x \in \F^n.
\end{equation}
 Since $\beta$ is a basis of $\L_P$, then there exists $C = (c_{i\,j}) \in \M_{d\,d}(\F)$ such that 
\begin{equation}\label{lem:PPhiPsi:eq2}
l_{P, \phi(\mathbf a_i)} = \sum_{j = 1}^{d} c_{i\,j} l_{P, \mathbf a_j}\;\;\mbox{for all}\;\; 1 \le i \le d.
\end{equation}
Hence, by substituting~\eqref{lem:PPhiPsi:eq2} into \eqref{lem:PPhiPsi:eq1} we obtain that
\begin{equation}\label{lem:PPhiPsi:eq3}
\begin{pmatrix}
l_{P,\mathbf a_1}(\mathbf x)\\
\vdots\\
l_{P,\mathbf a_d}(\mathbf x)
\end{pmatrix}
= C 
\begin{pmatrix}
l_{P,\mathbf a_1}(\psi(\mathbf x))\\
\vdots\\
l_{P,\mathbf a_d} (\psi(\mathbf x))
\end{pmatrix}
\;\;\mbox{for all}\;\;\mathbf x \in \F^n.
\end{equation}
By substituting $\mathbf e_{1}, \ldots, \mathbf e_{n}$ into~\eqref{lem:PPhiPsi:eq3} instead of $\mathbf x$ we conclude that 
\begin{equation*}
\begin{pmatrix}
1\\
\vdots\\
0
\end{pmatrix}, \ldots, 
\begin{pmatrix}
0\\
\vdots\\
1
\end{pmatrix} \in \colsp (C).
\end{equation*}
Hence, $\rk(C) = d$ and consequently $C$ is invertible matrix. By multiplying~\eqref{lem:PPhiPsi:eq3} by $C$ we have
\[
C^{-1}\begin{pmatrix}
l_{P,\mathbf a_1}(\mathbf x)\\
\vdots\\
l_{P,\mathbf a_d}(\mathbf x)
\end{pmatrix}
=  
\begin{pmatrix}
l_{P,\mathbf a_1}(\psi(\mathbf x))\\
\vdots\\
l_{P,\mathbf a_d} (\psi(\mathbf x))
\end{pmatrix}
\;\;\mbox{for all}\;\;\mathbf x \in \F^n.
\]
This implies that the function $l_{P,\mathbf a_i}\circ \psi$ is a linear combination of $l_{P,\mathbf a_1}, \ldots, l_{P,\mathbf a_d}$ for all $1 \le i \le d$. Since $l_{P,\mathbf a_i}(\rad(P)) = \{0\}$ for all $1 \le i \le d$ , then $\left(l_{P,\mathbf a_i}\circ \psi\right)(\rad(P)) = \{0\}.$ In addition, $l_{P,\mathbf a_i}\circ \psi$ is a linear function for all $1 \le i \le d$. Since  $\tilde{\psi}$ is a linear combination of $l_{P,\mathbf a_1}\circ \psi, \ldots, l_{P,\mathbf a_d}\circ \psi$ by its definition in~\eqref{lem:PPhiPsi:eq4}, then $\tilde{\psi}$ is a linear map such that $\tilde{\psi}(\rad(P)) = \{\mathbf 0\}$.

\paragraph{Claim~\ref{lem:PPhiPsi:part2}} Indeed,
\begin{multline}\label{lem:PPhiPsi:eq5}
l_{P,\mathbf a_i}(\psi (\mathbf x) - \tilde{\psi}(\mathbf x)) = l_{P,\mathbf a_i}(\psi (\mathbf x)) - l_{P,\mathbf a_i}\left( l_{P, \mathbf a_1} \left(\psi(\mathbf x)\right)\mathbf e_1 +  \ldots + l_{P, \mathbf a_d} \left(\psi(\mathbf x)\right)\mathbf e_d\right)\\
 \overset{\eqref{lem:PPhiPsi:eq6}}{=\joinrel=} l_{P,\mathbf a_i}(\psi (\mathbf x)) - l_{P,\mathbf a_i}(\psi (\mathbf x)) = 0\;\;\mbox{for all}\;\;\mathbf x \in \F^n\;\;\mbox{and}\;\; 1 \le i \le d.
\end{multline}
Since $\beta$ is a basis of $\L_P$, then~\eqref{lem:PPhiPsi:eq5} implies that  $l(\psi (\mathbf x) - \tilde{\psi}(\mathbf x)) = 0$ for all $l \in \L_P$. Hence, 
\begin{equation}\label{lem:PPhiPsi:eq8}
\left(\psi (\mathbf x) - \tilde{\psi}(\mathbf x)\right) \in \rad(P)
\end{equation}
by Lemma~\ref{lem:dzeroeqk}\ref{lem:dzeroeqk:part2} and consequently the equality~\eqref{lem:PPhiPsi:eqq6} holds. 

\paragraph{Claim~\ref{lem:PPhiPsi:part3}} Let $\psi_{\rad(P)} \colon \F^n / \rad(P) \to \F^n /\rad(P)$ be defined by
\begin{equation}\label{lem:PPhiPsi:eq9}
\psi_{\rad(P)} \left(\pi_{\rad(P)}(\mathbf x)\right) = \pi_{\rad(P)} \left( \tilde{\psi}(\mathbf x)\right)\quad\mbox{for all}\;\;\mathbf x \in \F^n.
\end{equation}
This definition is correct by~\eqref{lem:PPhiPsi:eqq4}. Since $\pi_{\rad(P)}$ and $\tilde{\psi}$ are linear maps, then $\psi_{\rad(P)}$ is linear as well. The equality~\eqref{lem:PPhiPsi:eq9} also implies the uniqueness of $\psi_{\rad(P)}$.

\paragraph{Claim~\ref{lem:PPhiPsi:part4}} Let $\psi_{\rad(P)} \colon \F^n / \rad(P) \to \F^n /\rad(P)$ be a linear map satisfying the equation~\eqref{lem:PPhiPsi:eq}. We show that $\psi_{\rad(P)}$ injective. Since $\F^n / \rad(P)$ is a finite dimensional vector space, then injectivity of linear map on $\F^n / \rad(P)$ implies its bijectivity. Accordingly, the injectivity of $\psi_{\rad(P)}$ establishes the statement of the claim. 

Let $\mathbf x \in \F^n$ be such that 
\begin{equation}\label{lem:PPhiPsi:eq10}
\psi_{\rad(P)}(\pi_{\rad(P)}(\mathbf x)) = \mathbf 0
\end{equation}
We show that $\mathbf x \in \rad(P)$. Indeed, let $\mathbf a \in \F^n$ and $\lambda \in \F$. Since $\phi$ and $\psi$ satisfy the condition~\eqref{lem:PPhiPsi:eqq1},
\begin{equation}\label{lem:PPhiPsi:eq12}
P(\mathbf a + \lambda \mathbf x) = P(\phi(\mathbf a) + \lambda \psi(\mathbf x)).
\end{equation}
The equality~\eqref{lem:PPhiPsi:eq} implies that
\begin{equation}\label{lem:PPhiPsi:eq11}
\pi_{\rad(P)} \left(\psi(\mathbf x)\right) = \psi_{\rad{P}} \left(\pi_{\rad(P)}(\mathbf x)\right).
\end{equation}
Hence, by substituting~\eqref{lem:PPhiPsi:eq10} into~\eqref{lem:PPhiPsi:eq11}, we obtain that
$\pi_{\rad(P)} \left(\psi(\mathbf x)\right) = \mathbf 0,$
or, equivalently,
$
\psi(\mathbf x)  \in \rad(P).
$
Hence
\begin{equation}\label{lem:PPhiPsi:eq13}
P(\phi(\mathbf a) + \lambda \psi(\mathbf x)) \overset{\mbox{\scriptsize The definition of}\;\rad(P)}{=\joinrel=\joinrel=\joinrel=\joinrel=\joinrel=\joinrel=\joinrel=\joinrel=\joinrel=\joinrel=\joinrel=\joinrel=\joinrel=\joinrel=} P(\phi(\mathbf a)) \overset{\mbox{\scriptsize Lemma~\ref{lem:condimpliesphi}}}{=\joinrel=\joinrel=\joinrel=\joinrel=\joinrel=\joinrel=} P(\mathbf a)
\end{equation}
By aligning~\eqref{lem:PPhiPsi:eq12} and~\eqref{lem:PPhiPsi:eq13} together, we obtain the equality
\[
P(\mathbf a + \lambda \mathbf x) = P(\mathbf a).
\]
This equality holds for all $\mathbf a \in \F^n$ and $\lambda \in \F$. Hence, $\mathbf x \in \rad(P)$ by the definition of $\rad(P)$. Hence, $\pi_{\rad(P)}(\mathbf x) = \mathbf 0$. 

Thus, the injectivity of $\psi_{\rad(P)}$ is established.
\end{proof}

If we assume that $\psi = \phi$, then we can make the following remarkable conclusion.

\begin{corollary}Let $P \in \F[x_1,\ldots, x_n]$ be a polynomial such that $\deg (P) < |\F|$ and $\dim (\mathcal L_P) + \dim (\rad(P)) = n$.  If  $\phi\colon \F^n \to \F^n$ is such that
\begin{equation*}
P(\mathbf{x} + \lambda \mathbf{y}) = P(\phi(\mathbf{x}) + \lambda \phi(\mathbf{y}))\;\;\mbox{for all}\;\;\mathbf{x},\mathbf{y} \in \mathbb \F^n,\; \lambda \in \F,
\end{equation*}
then there exists a unique linear map $T_{\rad} \colon \F^n/\rad(P) \to \F^n/\rad(P)$ preserving $P_{\rad(P)}$ such that
\begin{equation*}
\pi_{\rad(P)} \circ \phi = T_{\rad(P)} \circ \pi_{\rad(P)}. 
\end{equation*}
\end{corollary}

\begin{remark}
If we do not assume that $\psi = \phi$, then we cannot conclude from the statement of Lemma~\ref{lem:condimpliesphi} that $\psi_{\rad(P)}$ should preserve $P$ or $\pi_{\rad(P)} \circ \phi$ should be linear. Let $P = x_1(x_2 + 1) \in \F[x_1, x_2].$ Let us choose some $k \in \F$ such that $k \neq 0, 1$ and consider maps $\phi((x_1,x_2)) = (\frac{1}{k}x_1,k x_2 + k - 1), \psi ((x_1,x_2)) = (\frac{1}{k}x_1, kx_2).$

As it is stated in the theorem below, the discussed conclusions could be made if we assume that $P$ is homogeneous.
\end{remark}

\begin{theorem}\label{thm:homogenphipsi}Let $P \in \F[x_1,\ldots, x_n]$ be a homogeneous polynomial such that $\deg (P) < |\F|$ and $\dim (\mathcal L_P) + \dim (\rad(P)) = n$.  If  $\phi, \psi \colon \F^n \to \F^n$ are such that
\begin{equation}\label{thm:homogenphipsi:eqq1}
P(\mathbf{x} + \lambda \mathbf{y}) = P(\phi(\mathbf{x}) + \lambda \psi(\mathbf{y}))\;\;\mbox{for all}\;\;\mathbf{x},\mathbf{y} \in \mathbb \F^n,\; \lambda \in \F,
\end{equation}
then there exists a unique linear map $T_{\rad} \colon \F^n/\rad(P) \to \F^n/\rad(P)$ preserving $P_{\rad(P)}$ such that
\begin{equation*}\label{thm:homogenphipsi:eq}
\pi_{\rad(P)} \circ \phi = \pi_{\rad(P)} \circ \psi = T_{\rad(P)} \circ \pi_{\rad(P)}. 
\end{equation*}
\end{theorem} 
\begin{proof}Let $P \in \F[x_1,\ldots, x_n]$ be polynomial and $\phi, \psi\colon \F^n \to \F^n$ be maps satisfying all the conditions of the lemma.  Let $\psi_{\rad}\colon \F^n / \rad(P) \to \F^n / \rad(P)$ be a bijective linear map satisfying the equality~\eqref{lem:PPhiPsi:eq}. Such a linear map exists by Lemma~\eqref{lem:PPhiPsi}. The following five claims will be proven sequently:
\begin{enumerate}[label=(\roman*), ref=(\roman*)]
\item\label{thm:homogenphipsi:part1} If $\deg(P) > 1$, then $\phi(\mathbf 0) \in \rad(P)$.
\item\label{thm:homogenphipsi:part2} $\psi$ preserves $P$.
\item\label{thm:homogenphipsi:part3} $\psi_{\rad(P)}$ preserves $P_{\rad}$.
\item\label{thm:homogenphipsi:part4} $\pi_{\rad(P)} \circ \phi = \pi_{\rad(P)} \circ \psi$.
\item\label{thm:homogenphipsi:part5} $\psi_{\rad(P)}$ is a unique linear map satisfying the conditions of the lemma.
\end{enumerate}
The claim~\ref{thm:homogenphipsi:part5} clearly implies the statement of the lemma.
\paragraph{Claim~\ref{thm:homogenphipsi:part1}}
Assume in that  $d_P = \deg(P) > 1$. Let us ensure that $\phi(\mathbf 0) \in \rad(P)$ using Lemma~\ref{lem:StrangeCondImplyRadP}. For this, let us show that
\begin{equation}\label{thm:homogenphipsi:eq77}
P(\mathbf a + \lambda \phi (\mathbf 0)) = P(\mathbf a + \phi (\mathbf 0))\;\;\mbox{for all}\;\;\mathbf a \in \F^n\;\;\mbox{and}\;\;0\neq \lambda \in \F.
\end{equation}

Let $\mathbf a \in \F^n$ and $0 \neq \lambda \in \F$. Since we assume that $P$ is homogeneous,
\begin{equation}\label{thm:homogenphipsi:eq5}
P\left(\mathbf a + \lambda \phi (\mathbf 0)\right) = \lambda^{d_P} P\left(\frac{1}{\lambda}\mathbf a + \phi (\mathbf 0)\right).
\end{equation}
The map $\psi_{\rad(P)}$ is bijective. Consequently, there exists $\mathbf b \in \F^n$ such that 
\begin{equation}\label{thm:homogenphipsi:eq1}
\psi_{\rad(P)}\left(\pi_{\rad(P)}(\mathbf b)\right) = \pi_{\rad(P)}(\mathbf a).
\end{equation}
Hence,
\begin{equation*}
\pi_{\rad(P)} \left(\psi(\mathbf b)\right) \overset{\eqref{lem:PPhiPsi:eq}}{=\joinrel=} \psi_{\rad(P)}\left(\pi_{\rad(P)}(\mathbf b)\right)  \overset{\eqref{thm:homogenphipsi:eq1}}{=\joinrel=} \pi_{\rad(P)}(\mathbf a),
\end{equation*}
or, equivalently
\begin{equation}\label{thm:homogenphipsi:eq2}
\left(\psi (\mathbf b) - \mathbf a\right) \in \rad(P). 
\end{equation}
Therefore,
\begin{multline}\label{thm:homogenphipsi:eq6}
P(\frac{1}{\lambda}\mathbf a + \phi (\mathbf 0))\\ \overset{\mbox{\scriptsize The definition of}\;\rad(P)\;\mbox{\scriptsize and~\eqref{thm:homogenphipsi:eq2}}}{=\joinrel=\joinrel=\joinrel=\joinrel=\joinrel=\joinrel=\joinrel=\joinrel=\joinrel=\joinrel=\joinrel=\joinrel=\joinrel=\joinrel=\joinrel=\joinrel=\joinrel=\joinrel=} P\left(\frac{1}{\lambda} \left(\psi (\mathbf b) - \mathbf a\right) + \frac{1}{\lambda}\mathbf a + \phi (\mathbf 0)\right)\\ = P\left(\frac{1}{\lambda}\psi (\mathbf b) + \phi (\mathbf 0)\right)
\end{multline}
Since $\phi$ and $\psi$ satisfy the condition~\eqref{thm:homogenphipsi:eqq1}, then
\begin{equation}\label{thm:homogenphipsi:eq3}
P\left(\frac{1}{\lambda}\psi (\mathbf b) + \phi (\mathbf 0)\right) = P\left(\frac{1}{\lambda}\mathbf b + \mathbf 0\right) = \frac{1}{\lambda^{d_P}}P\left(\mathbf b\right)
\end{equation}
and
\begin{multline}\label{thm:homogenphipsi:eq4}
P\left(\mathbf b\right) = P\left(\mathbf 0 + \mathbf b\right) = P\left(\phi(\mathbf 0) + \psi(\mathbf b)\right)\\ \overset{\mbox{\scriptsize The definition of}\;\rad(P)\;\mbox{\scriptsize and~\eqref{thm:homogenphipsi:eq2}}}{=\joinrel=\joinrel=\joinrel=\joinrel=\joinrel=\joinrel=\joinrel=\joinrel=\joinrel=\joinrel=\joinrel=\joinrel=\joinrel=\joinrel=\joinrel=\joinrel=\joinrel=\joinrel=}
P\left(\phi(\mathbf 0) + \psi(\mathbf b) + \left(\mathbf a - \psi(\mathbf b)\right)\right)\\
=  P\left(\phi(\mathbf 0) + \mathbf a\right).
\end{multline}

Thus,
\begin{multline*}
P\left(\mathbf a + \lambda \phi (\mathbf 0)\right) \overset{\eqref{thm:homogenphipsi:eq5}}{=\joinrel=}  \lambda^{d_P} P\left(\frac{1}{\lambda}\mathbf a + \phi (\mathbf 0)\right)\overset{\eqref{thm:homogenphipsi:eq6}}{=\joinrel=}\\
\lambda^{d_P} P\left(\frac{1}{\lambda}\psi (\mathbf b) + \phi (\mathbf 0)\right) \overset{\eqref{thm:homogenphipsi:eq3}}{=\joinrel=} \lambda^{d_P} \cdot \frac{1}{\lambda^{d_P}}P\left(\mathbf b\right) \overset{\eqref{thm:homogenphipsi:eq4}}{=\joinrel=} P\left(\phi(\mathbf 0) + \mathbf a\right),
\end{multline*}
which holds for all $\mathbf a \in \F^n$ and $0 \neq \lambda \in \F$. From this we conclude that~\eqref{thm:homogenphipsi:eq77} holds for $\phi(\mathbf 0)$. Since $\deg(P) > 1$, then $|\F| > \deg(P)$ implies that $|\F| > 2$. Hence, the conditions Lemma~\ref{lem:StrangeCondImplyRadP} are satisfied for $\F$ and $\phi(\mathbf 0)$  and consequently $\phi(\mathbf 0) \in \rad(P).$

\paragraph{Claim~\ref{thm:homogenphipsi:part2}}
Let us consider three cases: (I) $\deg(P) < 1$; (II) $\deg(P) = 1$ and (III) $\deg(P) > 1$.

\paragraph{Case (I): $\deg(P) < 1$} Then $P$ is a constant function. Therefore, $\phi$ tautologically preserves $P$. 

\paragraph{Case (II): $\deg(P) = 1$} Then $P$ is a linear function. Therefore, the condition~\eqref{thm:homogenphipsi:eqq1} can be rewritten as follows
\begin{multline}\label{thm:homogenphipsi:eq7}
P(\mathbf x) + \lambda P(\mathbf y) = P(\mathbf x + \lambda \mathbf y) = P(\phi(\mathbf x) + \lambda \psi(\mathbf y))\\ = P(\phi(\mathbf x)) + \lambda P(\psi(\mathbf y))\;\;\mbox{for all}\;\;\mathbf x, \mathbf y \in \F^n, \lambda \in \F.
\end{multline}
Since Lemma~\ref{lem:condimpliesphi} implies that
\begin{equation}\label{thm:homogenphipsi:eq8}
P(\mathbf x) = P(\phi(\mathbf x))\;\;\mbox{for all}\;\;\mathbf x \in \F^n,
\end{equation}
by subtracting~\eqref{thm:homogenphipsi:eq8} from~\eqref{thm:homogenphipsi:eq7} we obtain that
\begin{equation*}
\lambda P(\mathbf y) = \lambda P(\psi(\mathbf y))\;\;\mbox{for all}\;\;\mathbf y \in \F^n, \lambda \in \F.
\end{equation*}
In particular, for $\lambda = 1$ we obtain that
\begin{equation*}
P(\mathbf y) = P(\psi(\mathbf y))\;\;\mbox{for all}\;\;\mathbf y \in \F^n.
\end{equation*}
Thus, $\psi$ preserves $P$.

\paragraph{Case (III): $\deg(P) > 1$} Then Claim~\ref{thm:homogenphipsi:part1} implies that 
\begin{equation}\label{thm:homogenphipsi:eq9}
\phi(\mathbf 0) \in \rad(P).
\end{equation}
Since $\phi$ and $\psi$ satisfy the condition~\eqref{thm:homogenphipsi:eqq1},
\begin{multline}\label{thm:homogenphipsi:eq10}
 P(\phi(\mathbf 0) + \psi(\mathbf x)) = P(\phi(\mathbf 0) + 1\cdot \psi(\mathbf x)) = P(\mathbf 0 + 1\cdot \mathbf x)\\ = P(\mathbf 0 + \mathbf x) = P(\mathbf x)\;\;\mbox{for all}\;\;\mathbf x \in \F^n.
\end{multline}
Therefore,
\[
P(\psi(\mathbf x)) \overset{\eqref{thm:homogenphipsi:eq9}}{=\joinrel=}  P(\phi(\mathbf 0) + \psi(\mathbf x))\overset{\eqref{thm:homogenphipsi:eq10}}{=\joinrel=} P(\mathbf x)\;\;\mbox{for all}\;\;\mathbf x \in \F^n,
\]
which means that $\psi$ preserves $P$.
\paragraph{Claim~\ref{thm:homogenphipsi:part3}}
The proposition that $\psi$ preserves $P$ could be rewritten as follows
\begin{equation}\label{thm:homogenphipsi:eq1001}
P \circ \psi = P.
\end{equation}
In addition,
\begin{equation}\label{thm:homogenphipsi:eq11}
P_{\rad(P)} \circ  \pi_{\rad(P)} = P
\end{equation}
by the definition of $P_{\rad(P)}$ and 
\begin{equation}\label{thm:homogenphipsi:eq12}
\psi_{\rad(P)} \circ \pi_{\rad(P)} = \pi_{\rad(P)} \circ \psi
\end{equation}
by the definition of $\psi_{\rad(P)}$. Hence,
\begin{multline}\label{thm:homogenphipsi:eq13}
\left(P_{\rad(P)} \circ  \psi_{\rad(P)}\right) \circ \pi_{\rad(P)} = P_{\rad(P)} \circ  \left(\psi_{\rad(P)}\circ \pi_{\rad(P)}\right)\\
\overset{\eqref{thm:homogenphipsi:eq12}}{=\joinrel=} P_{\rad(P)} \circ  \left(\pi_{\rad(P)} \circ \psi \right)
= \left(P_{\rad(P)} \circ  \pi_{\rad(P)} \right) \circ \psi\\ \overset{\eqref{thm:homogenphipsi:eq11}}{=\joinrel=} P \circ \psi \overset{\eqref{thm:homogenphipsi:eq1001}}{=\joinrel=} P \overset{\eqref{thm:homogenphipsi:eq11}}{=\joinrel=} P_{\rad(P)} \circ  \pi_{\rad(P)}.
\end{multline}
Since $\pi_{\rad(P)}$ is a surjective map, then~\eqref{thm:homogenphipsi:eq13} implies that
\[
P_{\rad(P)} \circ  \psi_{\rad(P)} = P_{\rad(P)}.
\]
Thus, $\psi_{\rad(P)}$ preserves $P_{\rad(P)}$.

\paragraph{Claim~\ref{thm:homogenphipsi:part4}}
This is equivalent to the following
\begin{equation*}\label{thm:homogenphipsi:eq14}
\left(\phi(\mathbf x) - \psi(\mathbf x)\right) \in \rad(P)\;\;\mbox{for all}\;\;\mathbf x \in \F^n,
\end{equation*}
which is equivalent to that
\begin{equation}\label{thm:homogenphipsi:eq15}
P(\mathbf a + \lambda \left(\phi(\mathbf x) - \psi(\mathbf x)\right)) = P(\mathbf a)\;\;\mbox{for all}\;\;\mathbf a, \mathbf x \in \F^n, \lambda \in \F.
\end{equation}
Let $\mathbf a, \mathbf x \in \F^n$ and $\lambda  \in \F.$ Then
{\setlength{\jot}{15pt}
\begin{multline*}\label{thm:homogenphipsi:eq16}
P\left(\mathbf a + \lambda\left(\phi(\mathbf x) - \psi(\mathbf x)\right)\right)\\
\overset{\scriptsize\mbox{Definition of $P_{\rad(P)}$}}{=\joinrel=\joinrel=\joinrel=\joinrel=\joinrel=\joinrel=\joinrel=\joinrel=\joinrel=\joinrel=\joinrel=}
P_{\rad(P)} \circ \pi_{\rad(P)}\left[\mathbf a + \lambda\left(\phi(\mathbf x) - \psi(\mathbf x)\right)\right]\\
\overset{\scriptsize\mbox{$\pi_{\rad(P)}$ is linear}}{=\joinrel=\joinrel=\joinrel=\joinrel=\joinrel=\joinrel=\joinrel=\joinrel=\joinrel=\joinrel=\joinrel=}
P_{\rad(P)} \left[\pi_{\rad(P)}(\mathbf a) + \pi_{\rad(P)}( \lambda \phi(\mathbf x)) - \lambda\pi_{\rad(P)}( \psi(\mathbf x))\right]\\
\overset{\scriptsize\mbox{$\psi_{\rad(P)}$ is invertible}}{=\joinrel=\joinrel=\joinrel=\joinrel=\joinrel=\joinrel=\joinrel=\joinrel=\joinrel=\joinrel=\joinrel=}
P_{\rad(P)} \left[\psi_{\rad(P)}\left(\psi_{\rad(P)}^{-1}\left(\pi_{\rad(P)}(\mathbf a)\right)\right)\right.\phantom{XXXXXXXXX}\\\phantom{XXXXXXXXXXXXXXXXXX}\left.
+ \pi_{\rad(P)}(\lambda  \phi(\mathbf x)) - \lambda \pi_{\rad(P)}( \psi(\mathbf x))\right]\\
\overset{\scriptsize\mbox{Definition of $\psi_{\rad(P)}$}}{=\joinrel=\joinrel=\joinrel=\joinrel=\joinrel=\joinrel=\joinrel=\joinrel=\joinrel=\joinrel=\joinrel=}
P_{\rad(P)} \left[\psi_{\rad(P)}\left(\psi_{\rad(P)}^{-1}\left(\pi_{\rad(P)}(\mathbf a)\right)\right)\right.\phantom{XXXXXXXXX}\\\phantom{XXXXXXXXXXXXXXX}\left. + \pi_{\rad(P)}( \lambda \phi(\mathbf x)) - \lambda \psi_{\rad(P)}\left(\pi_{\rad(P)}(\mathbf x)\right)\right]\\
\overset{\scriptsize\mbox{Rearrangement of summands}}{=\joinrel=\joinrel=\joinrel=\joinrel=\joinrel=\joinrel=\joinrel=\joinrel=\joinrel=\joinrel=\joinrel=\joinrel=\joinrel=\joinrel=\joinrel=\joinrel=}
P_{\rad(P)} \left[\psi_{\rad(P)}\left(\psi_{\rad(P)}^{-1}\left(\pi_{\rad(P)}(\mathbf a)\right)\right)\right.\phantom{XXXXXXX}\\\phantom{XXXXXXXXXXXXXXX}\left. - \lambda \psi_{\rad(P)}\left(\pi_{\rad(P)}(\mathbf x)\right) + \pi_{\rad(P)}( \lambda \phi(\mathbf x)) \right]\\
\overset{\scriptsize\mbox{$\psi_{\rad(P)}$ is linear}}{=\joinrel=\joinrel=\joinrel=\joinrel=\joinrel=\joinrel=\joinrel=\joinrel=\joinrel=\joinrel=\joinrel=}
P_{\rad(P)} \left[\psi_{\rad(P)}\left(\psi_{\rad(P)}^{-1}\left(\pi_{\rad(P)}(\mathbf a)\right) - \lambda \pi_{\rad(P)}(\mathbf x)\right)\right.\phantom{XXX}\\\phantom{XXXXXXXXXXXXXXXXXXXXXXXXXX}\left. + \pi_{\rad(P)}( \lambda \phi(\mathbf x)) \right]\\
\overset{\scriptsize\mbox{Lemma~\ref{lem:TwoCondLift}}}{=\joinrel=\joinrel=\joinrel=\joinrel=\joinrel=\joinrel=\joinrel=\joinrel=\joinrel=\joinrel=\joinrel=}
P_{\rad(P)} \left[\left(\psi_{\rad(P)}^{-1}\left(\pi_{\rad(P)}(\mathbf a)\right) - \lambda \pi_{\rad(P)}(\mathbf x)\right) + \pi_{\rad(P)}( \lambda \mathbf x) \right]\\
\overset{\scriptsize\mbox{$\pi_{\rad(P)}$ is linear}}{=\joinrel=\joinrel=\joinrel=\joinrel=\joinrel=\joinrel=\joinrel=\joinrel=\joinrel=\joinrel=\joinrel=}
P_{\rad(P)} \left[\psi_{\rad(P)}^{-1}\left(\pi_{\rad(P)}(\mathbf a)\right) - \lambda \pi_{\rad(P)}(\mathbf x) + \lambda \pi_{\rad(P)}( \mathbf x) \right]\\ 
=P_{\rad(P)} \left[\psi_{\rad(P)}^{-1}\left(\pi_{\rad(P)}(\mathbf a)\right) \right]\phantom{XXXXXX}\\
\overset{\scriptsize\mbox{$\psi_{\rad(P)}$ preserves $P_{\rad(P)}$}}{=\joinrel=\joinrel=\joinrel=\joinrel=\joinrel=\joinrel=\joinrel=\joinrel=\joinrel=\joinrel=\joinrel=\joinrel=\joinrel=\joinrel=}
P_{\rad(P)} \circ \psi_{\rad(P)} \left[\psi_{\rad(P)}^{-1}\left(\pi_{\rad(P)}(\mathbf a)\right) \right]\\
= P_{\rad(P)} \circ \pi_{\rad(P)}(\mathbf a)
\overset{\scriptsize\mbox{Definition of $P_{\rad(P)}$}}{=\joinrel=\joinrel=\joinrel=\joinrel=\joinrel=\joinrel=\joinrel=\joinrel=\joinrel=\joinrel=\joinrel=}
P(\mathbf a).
\end{multline*}}
Thus, \eqref{thm:homogenphipsi:eq15} holds for all $\mathbf a, \mathbf x \in \F^n$, $\lambda  \in \F$ and consequently $\pi_{\rad(P)} \circ \phi = \pi_{\rad(P)} \circ \psi$.
\paragraph{Claim~\ref{thm:homogenphipsi:part5}}
Indeed, the map $\psi_{\rad(P)}$ is linear by Lemma~\ref{lem:PPhiPsi}, preserves $P_{\rad(P)}$ by the claim~\ref{thm:homogenphipsi:part3} and the following equalities hold
\[
\pi_{\rad(P)} \circ \phi 
\overset{\scriptsize\mbox{Claim~\ref{thm:homogenphipsi:part4}}}{=\joinrel=\joinrel=\joinrel=\joinrel=\joinrel=}
\pi_{\rad(P)} \circ \psi
\overset{\scriptsize\mbox{Lemma~\ref{lem:PPhiPsi}}}{=\joinrel=\joinrel=\joinrel=\joinrel=\joinrel=\joinrel=}
\psi_{\rad(P)} \circ \pi_{\rad(P)}. 
\]
The last equality implies that $\psi_{\rad(P)}$ is unique by Lemma~\ref{lem:PPhiPsi}.
\end{proof}

\begin{remark}\label{rem:replacingCond}It is not known for the author if is it possible to omit the condition  $\dim (\mathcal L_P) + \dim (\rad(P)) = n$ from the statement of Lemma~\ref{lem:PPhiPsi} and Theorem~\ref{thm:homogenphipsi}. Notably, this condition holds automatically in the case if $\Char (\F) = 0$ as it is formulated in Theorem~\ref{thm:homogenphipsichar0} below. 
\end{remark}

\begin{theorem}\label{thm:homogenphipsichar0}Assume that $\Char(\F) = 0$. Let $P \in \F[x_1,\ldots, x_n]$ be a homogeneous polynomial. If  $\phi, \psi \colon \F^n \to \F^n$ are such that
\begin{equation*}
P(\mathbf{x} + \lambda \mathbf{y}) = P(\phi(\mathbf{x}) + \lambda \psi(\mathbf{y}))\;\;\mbox{for all}\;\;\mathbf{x},\mathbf{y} \in \mathbb \F^n,\; \lambda \in \F,
\end{equation*}
then there exists a unique linear map  $T_{\rad} \colon \F^n/\rad(P) \to \F^n/\rad(P)$ preserving $P_{\rad(P)}$ such that
\begin{equation}\label{thm:homogenphipsichar0:eq}
\pi_{\rad(P)} \circ \phi = \pi_{\rad(P)} \circ \psi = T_{\rad(P)} \circ \pi_{\rad(P)}. 
\end{equation}
\end{theorem} 
\begin{proof}Since $\Char(\F) = 0$, then $\F$ is infinite. Hence, the condition $|\F| > \deg(P)$ is satisfied. The condition $\dim (\mathcal L_P) + \dim (\rad(P)) = n$ holds by Corollary~\ref{cor:charzerokp}. Therefore, $P$, $\phi$ and $\psi$ satisfy all the necessary conditions of Theorem~\ref{thm:homogenphipsi} and consequently the equality~\eqref{thm:homogenphipsichar0:eq} holds.
\end{proof}

The example below shows that in the case if $\Char (\F) > 0$ the condition $\dim (\mathcal L_P) + \dim (\rad(P)) = n$ is not necessary at least for certain polynomials.

\begin{example}\label{ex:conditionsretained}Assume that $\Char (\F) = p \neq 0$. Let $P = x^p \in \F[x].$ Then $\dim (\mathcal L_P) + \dim (\rad(P)) = 0,$ but every pair $(\phi, \psi)$ satisfying condition~\ref{thm:homogenphipsi:eqq1} also satisfy this condition for $Q = x \in F[x].$ Since $\dim (\mathcal L_Q) + \dim (\rad(Q)) = 1,$ Lemma~\ref{lem:PPhiPsi} and Theorem~\ref{thm:homogenphipsi} hold for $Q, \phi, \psi$ and therefore they also hold for $P, \phi, \psi.$
\end{example}

The following theorem is useful to generate new elements of $\mathcal L_P$ from the given one. In particular, it is employed in Section~\ref{sec:AppToCullisDet}.

\begin{lemma}\label{thm:surjpres}Let $P \in \F[x_1,\ldots, x_n]$ be a polynomial such that $\deg (P) < |\F|$. If  $\phi, \psi \colon \F^n \to \F^n$ are such that
\begin{equation*}
P(\mathbf{x} + \lambda \mathbf{y}) = P(\phi(\mathbf{x}) + \lambda \psi(\mathbf{y}))\;\;\mbox{for all}\;\;\mathbf{x},\mathbf{y} \in \mathbb \F^n,\; \lambda \in \F,
\end{equation*}
and $\phi$ is surjective, then $l \circ \psi \in \mathcal L_P$ for all $l \in \mathcal L_P.$
\end{lemma}

\begin{proof}Let $m = \dim(\mathcal L_P)$ and $l = \sum_{i=1}^m c_i l_{P, \mathbf a_i}(\mathbf{v}) \in \mathcal L_P.$ Since $\phi$ is surjective, then there exist $\mathbf{b}_1, \ldots, \mathbf{b}_m$ such that $\phi (\mathbf{b}_i) = \mathbf{a}_i$ for all $1 \le i \le m.$

Since
\begin{multline*}
P(\mathbf a_i + \lambda \psi(\mathbf v)) = P(\phi \mathbf (\mathbf a_i) + \lambda \psi (\mathbf v))\\
 = P(\mathbf b_i + \lambda \mathbf v)\;\;\mbox{for all}\;\;\mathbf v \in F^n, \lambda \in \F,\;1 \le i \le m,
\end{multline*}
 by Lemma~\ref{lem:eqpd} for $l \circ \psi$ we have
\[
l\circ \psi(\mathbf v) = \sum_{i=1}^m c_i l_{P,\mathbf a_i}(\psi (\mathbf v)) = \sum_{i=1}^m c_i l_{P,\mathbf b_i}(\mathbf v)\;\;\mbox{for all}\;\;\mathbf v \in \F^n. 
\]
Therefore, $l\circ \psi \in \mathcal L_P.$
\end{proof}

\section{Nonlinear maps preserving the Cullis' determinant}
\label{sec:AppToCullisDet}

In this section we demonstrate how the theory developed above is applied to the particular polynomial matrix invariants. Namely, in Theorem~\ref{thm:homogenphipsiDetNKEven} and Theorem~\ref{thm:homogenphipsiDetNKOdd} we provide the explicit description for maps $\phi$ and $\psi$ on the space $\M_{n\,k}(\F)$ of $n\times k$-rectangular matrices satisfying the condition
\begin{equation}\label{cond:DETNK}
\det_{n\,k}(X + \lambda Y) = \det_{n\,k}(\phi(X) + \lambda \psi(Y))\;\;\mbox{for all}\;\;X, Y \in \M_{n\,k}(\F)\;\;\mbox{and}\;\;\lambda \in \F,
\end{equation}
where $\det_{n\,k}$ denotes the Cullis' determinant (Definition~\ref{def:DETNK}). \bigskip

Let us introduce the notation used in this section. 

For  $1 \le i_0 \le n$, $1 \le j_0 \le k$  we denote by $\mathsf x_{i_0\,j_0} \in \M_{n\,k}^*(\F)$  a linear function defined by $\mathsf x_{i_0\,j_0}(X) = x_{i_0\,j_0}$ for all $X = (x_{i\,j}) \in \M_{n\,k}(\F)$.

For $A \in \M_{n\,k_1}(\F)$ and $B \in \M_{n\,k_2}(\F)$ by $A|B \in \M_{n\,(k_1+k_2)}(\F)$ we denote a block matrix defined by $A|B = \begin{pmatrix} A & B\end{pmatrix}$.

By $A[J_1|J_2]$ we denote the $|J_1| \times |J_2|$ submatrix of $A$ lying on the intersection of rows with the indices from $J_1$ and the columns with the indices from $J_2$. By $A(J_1|J_2)$ we denote a submatrix of $A$ derived from it by striking out from it the rows with indices belonging to $J_1$ and the columns with the indices belonging to $J_2$. If one of the two index sets is absent, then it means an empty set, i.e. $A(J_1|)$ denotes a matrix derived from $A$ by striking out from it the rows with indices belonging to $J_1$. We may skip curly brackets, i.e. $A[1,2|3,4] = A[\{1,2\}|\{3,4\}]$. The notation with mixed brackets is also used, i.e. $A(|1]$ denotes the first column of the matrix $A$. The above notations are also used for vectors as well. In this case vectors are considered as $n\times 1$ or $1 \times n$ matrices.

\begin{definition}\label{def:DETNK}
The Cullis' determinant $\det_{n\,k}$ is a polynomial function on $\M_{n\,k}(\F)$ defined as an alternating sum of basic minors of $X$. That is,
\begin{equation*}
\det_{n\,k}(X) = \sum_{1 < i_1 \ldots < i_k \le n} (-1)^{i_1 + \ldots + i_k - 1 - \ldots - k}\det (X[i_1,\ldots, i_k|))
\end{equation*}
for all $X \in \M_{n\,k}(\F).$ We also denote $\det_{n\, k} (X)$ as follows
\[
\det_{n\,k}(X)=\begin{vmatrix}x_{1\,1} & \cdots & x_{1\,k}\\
\vdots & \cdots & \vdots\\
x_{n\,1} & \cdots & x_{n\,k}
\end{vmatrix}_{n\,k}\;\;\mbox{for all}\;\;X = (x_{i\,j}) \in \M_{n\,k}(\F).
\]
In the case if $n = k$, then $\det_{n\,k}$ is also denoted as $\det_{k}$ and is clearly equal to a classical determinant of a square matrix.
\end{definition}

Let us list the properties of $\det_{n\,k}$ which are similar to corresponding properties of the ordinary determinant (see \cite[\textsection 5, \textsection 27, \textsection 32]{cullis1913} or~\cite{NAKAGAMI2007422} for detailed proofs).

\begin{theorem}[{\cite[Theorem 13, Theorem 16]{NAKAGAMI2007422}}]
\begin{enumerate}
\item[]
\item For $X \in \M_{n}(\mathbb F),$ $\det_{n\,n}(X) = \det (X).$
\item For $X \in \M_{n\,k}(\mathbb F),$ $\det_{n\,k}(X)$ is a linear function of columns of $X$.
\item If a matrix $X \in \M_{n\,k}(\mathbb F)$ has two identical columns or one of its columns is a linear combination of other columns, then $\det_{n\,k}(X)$ is equal to zero.
\item For $X \in \M_{n\,k}(\mathbb F),$ interchanging any two columns of $X$ changes the sign of $\det_{n\,k}(X)$.
\item Adding a linear combination of columns of $X$ to another column of $X$ does not change $\det_{n\,k}(X)$.
\item For $X \in \M_{n\,k}(\mathbb F),$ $\det_{n\,k}(X)$ can be calculated using the Laplace expansion along a column of $X$. 
\end{enumerate}
\end{theorem}

\begin{corollary}\label{cor:CullisBinomialExpansion}Let $n \ge k$, $A, B \in \M_{n\,k} (\F).$ 
Then
\begin{multline}\label{eq:CullisBinomialExpansion}
\det_{n\,k} (A + t B)\\
= \sum_{d = 0}^{k}t^d \left( \sum_{1 \le i_1 < \ldots < i_d \le k} \det_{n\,k}\Bigl(A(|1]\Big|\ldots \Big| B(|i_1] \Big| \ldots \Big| B(|i_d] \Big| \ldots \Big| A(|k] \Bigr)\right),
\end{multline}
where both sides of the equality are considered as formal polynomials in $t$, i.e. as elements of $\F[t]$.
\end{corollary}
\begin{proof}This is a direct consequence from the multilinearity of $\det_{n\,k}$ with respect to the columns of a matrix.
\end{proof}

\begin{corollary}\label{cor:DegDetABLEQK}If $A, B \in \M_{n\,k}(\F)$, then $\deg_t\left(\det_{n\,k} (A + t B)\right) \le k$.
\end{corollary}

%
Lemma~\ref{lem:THMNPLUSKODDKALL} and Lemma~\ref{lem:THMNPLUSKODDKALL} taken together provide the description of linear maps preserving $\det_{n\,k}$ for $k \ge 3$ and $n \ge k + 2$ and are used in the sequel.

\begin{lemma}[{\cite[Theorem~5.14]{Guterman2025} and~\cite[Theorem~5.5]{Guterman2025c}}]\label{thm:THMNPLUSKEVEN}\label{lem:THMNPLUSKEVENKALL}
Assume that $|\F| > k \ge 3, n \ge k + 2$ and  $n + k$ is even. Let $T\colon \M_{n\, k} (\F) \to \M_{n\, k} (\F)$ be a linear map. Then $\det_{n\, k} (T(X)) = \det_{n\,k}(X)$ for all $X \in \M_{n\, k} (\F)$ if and only if there exist $A \in \M_{n\, n}(\F)$ and $B \in \M_{k\, k}(\F)$ such that
\begin{equation}\label{thm:MainTheoremEvenKGe4:eq}
\det_{n\, k} \Bigl(A(|i_1,\ldots, i_k]\Bigr) \cdot \det_k \Bigl(B\Bigr) = (-1)^{i_1 + \ldots + i_k - 1 - \ldots - k}
\end{equation}
for all increasing sequences $1 \le i_1 < \ldots < i_k \le n $ and
\[
T(X) = AXB\;\;\mbox{for all}\;\; X \in \M_{n\, k} (\F).
\]
\end{lemma}

\begin{lemma}[{Cf.~\cite[Theorem~4.12]{Guterman2025b} and~\cite[Theorem~5.16]{Guterman2025c}}]\label{thm:THMNPLUSKODD}\label{lem:THMNPLUSKODDKALL}Assume that $|\F| > k \ge 3$, $n \ge k + 2$ and $n + k$ is odd. Let $T\colon \M_{n\,k} (\F) \to \M_{n\,k} (\F)$ be a linear map. Then $\det_{n\, k} (T(X)) = \det_{n\,k}(X)$ for all $X \in \M_{n\,k} (\F)$ if and only if  there exist $A \in \M_{n\, n}(\F)$ and $B \in \M_{k\, k}(\F)$ such that
\begin{equation}\label{thm:MainTheoremCullisNKOdd:eqq}
\det_{n\,k} \Bigl(A(|i_1,\ldots, i_k]\Bigr) \det_k \Bigl(B\Bigr) = (-1)^{i_1 + \ldots + i_k - 1 - \ldots - k}
\end{equation}
for all increasing sequences $1 \le i_1 < \ldots < i_k \le n $ and a linear map\linebreak $\omega\colon \M_{n\,k}(\F) \to W_{n\,k}$ such that
\begin{equation}\label{thm:MainTheoremCullisNKOdd:eq}
T(X) = AXB + \omega(X)\;\;\mbox{for all}\;\;X \in \M_{n\,k} (\F).
\end{equation}
\end{lemma}

As the reader can observe, the description of linear maps preserving $\det_{n\,k}$ depends on parity of $n + k$. Thus, we consider this cases separately.

\subsection{$n + k$ is even}

\begin{lemma}[{\cite[Lemma~4.1]{Guterman2025}}]\label{lem:DetNKXiEqX1}Assume that $n \ge k + 2$ and $n+k$ is even. Let $x_1, \ldots, x_n \in \F$ and
\[
X = \begin{pmatrix}
x_1 & 0 & 0 & \cdots & 0\\
x_2 & 1 & 0 &\cdots & 1\\
x_3 & 0 & 1 &\cdots & 1\\
\vdots & \vdots & \vdots & \ddots & \vdots\\
x_k & 0 & 0 & \cdots & 1\\
\vdots & \vdots & \vdots & \ddots & \vdots\\
x_{n} & 0 & 0 & \cdots & 1
\end{pmatrix} \in \M_{n\,k}(\F).
\]
Then
\[
\det_{n\,k}(X) = x_1.
\]
\end{lemma}

\begin{definition}[{Cf.~\cite[Definition~4.3]{Guterman2025}}]\label{def:SCSDefEven}Let $n \ge k$ and  $1 \le j \le k$. By $\SCS^{\mathrm{even}}_{i\,j} \colon \M_{n\,k}(\F) \to \M_{n\,k}(\F)$ we denote a linear map defined by
\begin{equation*}
\SCS_{i\,j}\begin{psmallmatrix}
x_{1\,1} & \cdots  & x_{1\,k}\\
  \vdots & \ddots & \vdots\\
x_{n\,1} & \cdots & x_{n\,k}\\
\end{psmallmatrix} =(-1)^{n-i}\cdot  
\begin{psmallmatrix}
(-1)^{1 - \delta_{1\,j}}x_{i\,j} & \cdots & x_{i\,1} & \cdots & x_{i\,k}\\
\vdots & \ddots & \vdots & \ddots & \vdots\\
(-1)^{1 - \delta_{1\,j}}x_{n\,j} & \cdots & x_{n\,1} & \cdots & x_{n\,k}\\
-(-1)^{1 - \delta_{1\,j}}x_{1\,j} & \cdots & -x_{1\,1} & \cdots & -x_{1\,k}\\
  \vdots & \ddots & \vdots & \ddots & \vdots\\
-(-1)^{1 - \delta_{1\,j}}x_{(i-1)\,j} & \cdots & -x_{(i-1)\,1} & \cdots &  -x_{(i-1)\,k}\\
\end{psmallmatrix}
\end{equation*}
for all $X = (x_{i\,j}) \in \M_{n\,k}(\F)$. That is, $\SCS_{i\,j}(X)$ is obtained from $X$ by performing the following sequence of operations:
\begin{enumerate}
\item the row cyclical shift sending $i$-th row of $X$ to the first row of the result;
\item multiplying the bottom $i-1$ rows by $-1$;
\item exchanging the first and the $j$-th column;
\item multiplying the first column  by $(-1)^{1 - \delta_{1\,j}}$;
\item multiplying all the entries by $(-1)^{n-i}$.
\end{enumerate}
\end{definition}

\begin{lemma}[{Cf.~\cite[Lemma~4.5]{Guterman2025}}]\label{lem:ReduceTo11ByShift}Assume that $n \ge k$ and $n + k$ be even. Then $\SCS^{\mathrm{even}}_{i\,j}$ is an invertible linear map preserving $\det_{n\,k}$ for all $1 \le i \le n$, $1 \le j \le n$.
\end{lemma}

\begin{lemma}\label{lem:NKEven11InL}Assume that  $n \ge k$ and  $n + k$ is even. Then $\mathsf x_{1\,1} \in \mathcal L_{\det_{n\,k}}$.
\end{lemma}

\begin{proof}The proof of this lemma is implicitly contained in the proof of Lemma~4.2 in~\cite{Guterman2025}. We provide an explicit proof for the clarity and convenience. 

Let us show that
\begin{equation}\label{lem:NKEven11InL:eq5}
l_{\det_{n\,k}, A} = \frac{\partial \det_{n\,k}}{\partial X}(A) = x_{1\,1}\;\;\mbox{for all}\;\; X = (x_{i\,j}) \in \M_{n\,k}(\F),
\end{equation}
where $A \in \M_{n\,k}(\F)$ is defined by
\[
A = E_{2\,2} + E_{3\,3} + \ldots + E_{(k-1)\,(k-1)} + E_{2\,k} + \ldots + E_{n\,k}  = \begin{pmatrix}
0 & 0 & 0 & \cdots & 0\\
0 & 1 & 0 &\cdots & 1\\
0 & 0 & 1 &\cdots & 1\\
\vdots & \vdots & \vdots & \ddots & \vdots\\
0 & 0 & 0 & \cdots & 1\\
\vdots & \vdots & \vdots & \ddots & \vdots\\
0 & 0 & 0 & \cdots & 1
\end{pmatrix}.
\]
This fact clearly implies the statement of the lemma.

Relying on Lemma~\ref{def:partDer}, we have the equality 
\begin{equation}\label{lem:NKEven11InL:eq2}
\frac{\partial \det_{n\,k}}{\partial X}(A) = g_{A, X}'(0),
\end{equation}
where $g_{A, X} = \det_{n\,k}(A + tX) \in \F[t]$. By Corollary~\ref{cor:DegDetABLEQK} we conclude that there exist $a_0, \ldots, a_k \in \F$ such that
 \[ g_{A, X} = \det_{n\, k} (A + t Y) = a_0 + a_1t + \ldots + a_k t^k.
  \]
The Definition~\ref{def:onedimder} implies that 
\begin{equation}\label{lem:NKEven11InL:eq1}
g_{A, X}'(0) = a_1.
\end{equation}
Using the expansion~\eqref{eq:CullisBinomialExpansion}, we obtain that 
\[
a_1 = \sum_{1 \le i_1 \le k} \det_{n\,k}\Bigl(A(|1]\Big|\ldots \Big| X(|i_1] \Big| \ldots \Big| A(|k] \Bigr).\]
Since $A(|1]$ is a zero column, we conclude by splitting off the first term that
\begin{multline}\label{lem:NKEven11InL:eq3}
a_1 = \sum_{1 \le i_1 \le k} \det_{n\,k}\Bigl(A(|1]\Big|\ldots \Big| X(|i_1] \Big| \ldots \Big| A(|k] \Bigr)\\
 = \det_{n\,k} \Bigl(X(|1] \Big| A(|2] \Big| \ldots \Big| A(|k]\Bigr)  + \sum_{2 \le i_1 \le k} \det_{n\,k}\Bigl(A(|1]\Big|\ldots \Big| X(|i_1] \Big| \ldots \Big| A(|k] \Bigr)\\
 = \det_{n\,k} \Bigl(X(|1] \Big| A(|2] \Big| \ldots \Big| A(|k]\Bigr)  + \sum_{2 \le i_1 \le k} \det_{n\,k}\Bigl(0\Big|\ldots \Big| X(|i_1] \Big| \ldots \Big| A(|k] \Bigr)\\
 =  \det_{n\,k} \Bigl(X(|1] \Big| A(|2] \Big| \ldots \Big| A(|k]\Bigr) + \sum_{2 \le i_1 \le k} 0
 = \det_{n\,k} \Bigl(X(|1] \Big| A(|2] \Big| \ldots \Big| A(|k]\Bigr).
\end{multline}

Now note that $X(|1] \Big| A(|2] \Big| \ldots \Big| A(|k]$ has the form described in the statement of Lemma~\ref{lem:DetNKXiEqX1} for $x_1 = x_{1\,1}, \ldots, x_n = x_{n\,1}$. This implies that\begin{equation}\label{lem:NKEven11InL:eq4}
\det_{n\,k} \Bigl(X(|1] \Big| A(|2] \Big| \ldots \Big| A(|k]\Bigr) = x_{1\,1}.
\end{equation}

By aligning the equalities~\eqref{lem:NKEven11InL:eq2}, \eqref{lem:NKEven11InL:eq1}, \eqref{lem:NKEven11InL:eq3} and~\eqref{lem:NKEven11InL:eq4}  together, we obtain the equality~\eqref{lem:NKEven11InL:eq5}, which concludes the proof.
\end{proof}

\begin{lemma}\label{lem:NKEvenEveryIJInL}Assume that  $n \ge k,$ $|\F| > k$ and  $n + k$ is even.  Then $\mathsf x_{i\,j} \in \mathcal L_{\det_{n\,k}}$ for all $1 \le i \le n$, $1 \le j \le k$.
\end{lemma}

\begin{proof}Lemma~\ref{lem:NKEven11InL} implies that $\mathsf x_{1\,1} \in \mathcal L_{\det_{n\,k}}$. Since $\SCS^{\mathrm{even}}_{i\,j}$ is a linear map preserving $\det_{n\,k}$, then
\begin{multline*}
\det_{n\,k}(X + \lambda Y) = \det_{n\,k}(\SCS^{\mathrm{even}}_{i\,j}(X + \lambda Y))\\
 = \det_{n\,k}(\SCS^{\mathrm{even}}_{i\,j}(X) + \lambda \SCS^{\mathrm{even}}_{i\,j}(Y)))\;\;\mbox{for all}\;\;X, Y \in \M_{n\,k}(\F), \lambda \in \F.
\end{multline*}
Therefore, 
\begin{equation}\label{lem:NKEvenEveryIJInL:eq1}
l =  \mathsf x_{1\,1} \circ \SCS^{\mathrm{even}}_{i\,j} \in \mathcal L_P
\end{equation}
  by Lemma~\ref{thm:surjpres}. Note that
\begin{equation*}
(\mathsf x_{1\,1} \circ \SCS^{\mathrm{even}}_{i\,j})(X) = (\SCS^{\mathrm{even}}_{i\,j})(X)_{1\,1} = (-1)^{1-\delta_{1\, j}}x_{i\,j}\;\;\mbox{for all}\;\;X \in \M_{n\,k}(\F)
\end{equation*}
by the definition of $\SCS^{\mathrm{even}}_{i\,j}$. Therefore,
\begin{equation}\label{lem:NKEvenEveryIJInL:eq2}
l = (-1)^{1-\delta_{1\, j}}\mathsf x_{i\,j}.
\end{equation}
Since $\L_{\det_{n\,k}}$ is a vector space, then~\eqref{lem:NKEvenEveryIJInL:eq1} and~\eqref{lem:NKEvenEveryIJInL:eq2} together imply that $\mathsf x_{i\,j} \in \L_{\det_{n\,k}}$. 
\end{proof}

\begin{corollary}\label{cor:DimNKEvenEqNK}Assume that  $n \ge k$, $|\F| > k$  and  $n + k$ is even. Then
\begin{enumerate}[label=(\alph*), ref=(\alph*)]
\item\label{cor:DimNKEvenEqNK:part1} $\L_{\det_{n\,k}} = \M_{n\,k}^*(\F)$;
\item\label{cor:DimNKEvenEqNK:part2} $\dim(\L_{\det_{n\,k}}) = nk$.
\end{enumerate}
\end{corollary}

Here we provide another proof of~\cite[Lemma~4.6]{Guterman2025} which employs the theory developed above.

\begin{lemma}\label{lem:NKEvenAdditionPreserverIsZero}Assume that $n \ge k$, $|\F| > k$ and $n + k$ is even. Then\linebreak $\rad(\det_{n\,k}) = \{0\}$.
\end{lemma}
\begin{proof}Indeed, $\dim(\L_{\det_{n\,k}}) + \dim(\rad(\det_{n\,k})) \le nk$ by Lemma~\ref{lem:dzeroeqk}\ref{lem:dzeroeqk:part1}. Since $\dim(\L_{\det_{n\,k}}) = nk$ by Corollary~\ref{cor:DimNKEvenEqNK}\ref{cor:DimNKEvenEqNK:part2}, $\dim(\rad(\det_{n\,k})) = 0$ and consequently $\rad(\det_{n\,k}) = \{0\}$.
\end{proof}

\begin{corollary}\label{cor:NKEvenSatDimCond}Assume that  $n \ge k$   and  $n + k$ is even. Then $\dim(\L_{\det_{n\,k}}) + \dim(\rad(\det_{n\,k})) = nk$.
\end{corollary}
%

\begin{theorem}\label{thm:homogenphipsiDetNKEven}Assume that  $n \ge k + 2,$ $|\F| > k$, $k \ge 3$ and  $n + k$ is even.  If  $\phi, \psi \colon \F^n \to \F^n$ are two maps satisfying the condition~(\ref{cond:DETNK}), then there exist  $A \in \M_{n\, n}(\F)$ and $B \in \M_{k\, k}(\F)$ such that
\begin{equation}\label{thm:homogenphipsiDetNKEven:eq}
\det_{n\, k} \Bigl(A(|i_1,\ldots, i_k]\Bigr) \cdot \det_k \Bigl(B\Bigr) = (-1)^{i_1 + \ldots + i_k - 1 - \ldots - k}
\end{equation}
for all increasing sequences $1 \le i_1 < \ldots < i_k \le n $ and
\begin{equation}\label{thm:homogenphipsiDetNKEven:eqq}
\phi(X) = \psi(X) = AXB\;\;\mbox{for all}\;\;X \in \M_{n\, k} (\F).
\end{equation}
\end{theorem} 
\begin{proof}The definition of $\det_{n\,k}$ implies that $\det_{n\,k}$ is a homogeneous polynomial of degree $k$. By Corollary~\ref{cor:NKEvenSatDimCond} we conclude that $\det_{n\,k}$, $\phi$ and $\psi$ satisfy the conditions of Theorem~\ref{thm:homogenphipsi}. Therefore, $\phi$ and $\psi$ preserve $\det_{n\,k}$ and $\phi = \psi$ because $\rad(\det_{n\,k}) = \{0\}$ by Lemma~\ref{lem:NKEvenAdditionPreserverIsZero}.
Thus, by Lemma~\ref{lem:THMNPLUSKEVENKALL}, there exist  $A \in \M_{n\, n}(\F)$ and $B \in \M_{k\, k}(\F)$ satisfying the condition\eqref{thm:homogenphipsiDetNKEven:eq} such that the equality\eqref{thm:homogenphipsiDetNKEven:eqq} holds.
\end{proof}

\subsection{$n + k$ is odd}

\begin{definition}[{\cite[Definition~3.6]{Guterman2025b}}]\label{def:WNK}By $W_{n\,k}\subseteq \M_{n\,k}(\F)$ we denote a $k$-dimensional vector space consisting of matrices, all rows of which are equal. That is,
\[
W_{n\,k} = \{\begin{pmatrix}y_1 & \cdots & y_k\\ \vdots & \ddots & \vdots\\ y_1 & \cdots & y_k\end{pmatrix} \mid y_1, \ldots, y_k \in \F\}.
 \]
\end{definition}

\begin{lemma}[{Cf.~\cite[Theorem~3.3]{amiri2010}}]\label{lem:AdditionNPLusKOdd}
Assume that $1 \le k \le n,$ and $k + n$ is an odd integer, $A, X \in \M_{n\, k}(\F)$ and $X = \begin{psmallmatrix}
x_1 & \ldots & x_k\\
\vdots & \ddots & \vdots\\
x_1 & \ldots & x_k
\end{psmallmatrix}$ for some $x_1,\ldots x_k \in \F$. Then
\[
\det_{n\,k} (A + X) = \det_{n\,k} (A).
\]
\end{lemma}

\begin{corollary}\label{cor:WNKIsInRadDetNK}Assume that $1 \le k \le n$ and $k + n$ is an odd integer. Then $W_{n\,k} \subseteq \rad_{n\,k}(\det_{n\,k}).$
\end{corollary}

\begin{definition}[{Cf.~\cite[Definition~3.4]{Guterman2025b}}]\label{def:SCSDefOdd}Let $n \ge k$ and $1 \le i \le n$, $1 \le j \le k$. By $\SCS^{\mathrm{odd}}_{i\,j}$ we denote a linear map on $\M_{n\,k}(\F)$ defined by 
\begin{equation*}
\SCS^{\mathrm{odd}}_{i\,j}\begin{psmallmatrix}
x_{1\,1} & \cdots  & x_{1\,k}\\
  \vdots & \ddots & \vdots\\
x_{n\,1} & \cdots & x_{n\,k}\\
\end{psmallmatrix}= (-1)^{i+1}\begin{psmallmatrix}
(-1)^{1 - \delta_{1\,j}}x_{i\, j} & \ldots & x_{i\,1} & \ldots & x_{i\, k}\\
\vdots & \ddots & \vdots & \ddots  & \vdots\\
(-1)^{1 - \delta_{1\,j}}x_{n\, j} & \ldots & x_{n\,1} & \ldots & x_{n\, k}\\
(-1)^{1 - \delta_{1\,j}}x_{1\, j} & \ldots &  x_{1\,1} & \ldots & x_{1\, k}\\
\vdots & \ddots & \vdots & \ddots & \vdots\\
(-1)^{1 - \delta_{1\,j}}x_{i-1\, j} & \ldots & x_{i-1\,1} & \ldots &  x_{i-1\, k}
\end{psmallmatrix}.
\end{equation*}
That is, $\SCS^{\mathrm{odd}}_{i\,j}(X)$ is obtained from $X$ by performing the following sequence of operations:
\begin{enumerate}
\item the row cyclical shift sending $i$-th row of $X$ to the first row of the result;
\item exchanging the first and the $j$-th column;
\item multiplying the first column  by $(-1)^{1 - \delta_{1\,j}}$;
\item multiplying all the entries by $(-1)^{i+1}$.
\end{enumerate}
\end{definition}

\begin{lemma}[{Cf.~\cite[Lemma~3.5]{Guterman2025b}}]\label{lem:ReduceTo11ByShiftOdd}Assume that $n \ge k$ and $n + k$ is odd. Then $\SCS^{\mathrm{odd}}_{i\,j}$ is an invertible linear map preserving $\det_{n\,k}$ for all $1 \le i \le n$, $1 \le j \le k$.
\end{lemma}
\begin{lemma}[{\cite[Lemma~3.2]{Guterman2025b}}]\label{lem:DetNKXiEqX1Odd}Suppose that $n \ge k \ge 1, n+k$ is odd,\linebreak $x_1, \ldots, x_n \in \F$ and
\[
X = \begin{pmatrix}
x_1 & 0 & 0 & \cdots & 0\\
x_2 & 0 & 0 & \cdots & 0\\
x_3 & 1 & 0 &\cdots & 1\\
x_4 & 0 & 1 &\cdots & 1\\
\vdots & \vdots & \vdots & \ddots & \vdots\\
x_k & 0 & 0 & \cdots & 1\\
\vdots & \vdots & \vdots & \ddots & \vdots\\
x_{n} & 0 & 0 & \cdots & 1
\end{pmatrix} \in \M_{n\,k}(\F).
\]
Then
\[
\det_{n\,k}(X) = (-1)^{k-1}(x_1 - x_2).
\]
\end{lemma}

\begin{lemma}\label{lem:NKOdd11InL}Assume that  $n \ge k$ and  $n + k$ is odd. Then $\mathsf x_{1\,1}  - \mathsf x_{2\,1} \in \mathcal L_{\det_{n\,k}}$.
\end{lemma}

\begin{proof}The proof of this lemma is implicitly contained in the proof of Lemma~3.3 in~\cite{Guterman2025b}. We provide an explicit proof for the clarity and convenience.

Let us show that
\begin{equation}\label{lem:NKOdd11InL:eq5}
l_{\det_{n\,k}, A} = \frac{\partial \det_{n\,k}}{\partial X}(A) = (-1)^{k-1}\left(x_{1\,1} - x_{2\,1}\right)\;\;\mbox{for all}\;\; X = (x_{i\,j}) \in \M_{n\,k}(\F),
\end{equation}
where $A \in \M_{n\, k}(\F)$ is defined by
\[
A = E_{3\,2} + E_{4\,3} + \ldots + E_{k\, k-1} + E_{3\, k} + \ldots + E_{n\, k}  = \begin{pmatrix}
0 & 0 & 0 & \cdots & 0\\
0 & 0 & 0 & \cdots & 0\\
0 & 1 & 0 &\cdots & 1\\
0 & 0 & 1 &\cdots & 1\\
\vdots & \vdots & \vdots & \ddots & \vdots\\
0 & 0 & 0 & \cdots & 1\\
\vdots & \vdots & \vdots & \ddots & \vdots\\
0 & 0 & 0 & \cdots & 1
\end{pmatrix}.
\]
This fact clearly implies the statement of the lemma.

Relying on Lemma~\ref{def:partDer}, we have the equality 
\begin{equation}\label{lem:NKOdd11InL:eq2}
\frac{\partial \det_{n\,k}}{\partial X}(A) = g_{A, X}'(0),
\end{equation}
where $g_{A, X} = \det_{n\,k}(A + tX) \in \F[t]$. By Corollary~\ref{cor:DegDetABLEQK} we conclude that there exist $a_0, \ldots, a_k \in \F$ such that
 \[ g_{A, X} = \det_{n\, k} (A + t Y) = a_0 + a_1t + \ldots + a_k t^k.
  \]
The Definition~\ref{def:onedimder} implies that 
\begin{equation}\label{lem:NKOdd11InL:eq1}
g_{A, X}'(0) = a_1.
\end{equation}
Using the expansion~\eqref{eq:CullisBinomialExpansion}, we obtain that
\[
a_1 = \sum_{1 \le i_1 \le k} \det_{n\,k}\Bigl(A(|1]\Big|\ldots \Big| X(|i_1] \Big| \ldots \Big| A(|k] \Bigr).
\]
Since $A(|1]$ is a zero column, we conclude by splitting off the first term that
\begin{multline}\label{lem:NKOdd11InL:eq3}
a_1 = \sum_{1 \le i_1 \le k} \det_{n\,k}\Bigl(A(|1]\Big|\ldots \Big| X(|i_1] \Big| \ldots \Big| A(|k] \Bigr)\\
 = \det_{n\,k} \Bigl(X(|1] \Big| A(|2] \Big| \ldots \Big| A(|k]\Bigr)  + \sum_{2 \le i_1 \le k} \det_{n\,k}\Bigl(A(|1]\Big|\ldots \Big| X(|i_1] \Big| \ldots \Big| A(|k] \Bigr)\\
 = \det_{n\,k} \Bigl(X(|1] \Big| A(|2] \Big| \ldots \Big| A(|k]\Bigr)  + \sum_{2 \le i_1 \le k} \det_{n\,k}\Bigl(0\Big|\ldots \Big| X(|i_1] \Big| \ldots \Big| A(|k] \Bigr)\\
 =  \det_{n\,k} \Bigl(X(|1] \Big| A(|2] \Big| \ldots \Big| A(|k]\Bigr) + \sum_{2 \le i_1 \le k} 0
 = \det_{n\,k} \Bigl(X(|1] \Big| A(|2] \Big| \ldots \Big| A(|k]\Bigr).
\end{multline}

Now note that $X(|1] \Big| A(|2] \Big| \ldots \Big| A(|k]$ has the form described in the statement of Lemma~\ref{lem:DetNKXiEqX1Odd} for $x_1 = x_{1\,1}, \ldots, x_n = x_{n\,1}$. This implies that\begin{equation}\label{lem:NKOdd11InL:eq4}
\det_{n\,k} \Bigl(X(|1] \Big| A(|2] \Big| \ldots \Big| A(|k]\Bigr) = (-1)^{k-1}\left(x_{1\,1} - x_{2\,1}\right).
\end{equation}
By aligning the equalities~\eqref{lem:NKOdd11InL:eq2}, \eqref{lem:NKOdd11InL:eq1}, \eqref{lem:NKOdd11InL:eq3} and~\eqref{lem:NKOdd11InL:eq4}  together, we obtain the equality~\eqref{lem:NKOdd11InL:eq5}, which concludes the proof.
\end{proof}

\begin{lemma}\label{lem:NKOddIJInL}Assume that  $n \ge k,$ $|\F| > k$ and  $n + k$ is odd. Then $(\mathsf x_{i\,j}  - \mathsf x_{(i+1)\,j}) \in \mathcal L_{\det_{n\,k}}$ for all $1 \le i < n$ and $1 \le j \le k$.
\end{lemma}
\begin{proof}
Lemma~\ref{lem:NKOdd11InL} implies that $\mathsf x_{1\,1} - \mathsf x_{2\,1} \in \mathcal L_{\det_{n\,k}}$. Since $\SCS^{\mathrm{odd}}_{i\,j}$ is a linear map preserving $\det_{n\,k}$, then
\begin{multline*}
\det_{n\,k}(X + \lambda Y) = \det_{n\,k}(\SCS^{\mathrm{odd}}_{i\,j}(X + \lambda Y))\\ = \det_{n\,k}(\SCS^{\mathrm{odd}}_{i\,j}(X) + \lambda \SCS^{\mathrm{odd}}_{i\,j}(Y)))\;\;\mbox{for all}\;\;X, Y \in \M_{n\,k}(\F), \lambda \in \F.
\end{multline*}
Therefore, 
\begin{equation}\label{lem:NKOddIJInL:eq1}
l =  (\mathsf x_{1\,1} - \mathsf x_{2\,1}) \circ \SCS^{\mathrm{odd}}_{i\,j} \in \mathcal L_P
\end{equation}
  by Lemma~\ref{thm:surjpres}. Note that
\begin{multline*}
(\mathsf x_{1\,1} - \mathsf x_{2\,1}) \circ \SCS^{\mathrm{odd}}_{i\,j}(X) = (\SCS^{\mathrm{odd}}_{i\,j})(X)_{1\,1} - (\SCS^{\mathrm{odd}}_{i\,j})(X)_{2\,1}\\ = (-1)^{i+1}
(-1)^{1 - \delta_{1\,j}}\left(x_{i\,j} - x_{(i+1)\,j}\right)\qquad \mbox{for all}\;\;X \in \M_{n\,k}(\F)
\end{multline*}
by the definition of $\SCS^{\mathrm{odd}}_{i\,j}$. Therefore,
\begin{equation}\label{lem:NKOddIJInL:eq2}
l = (-1)^{i+1}
(-1)^{1 - \delta_{1\,j}}(\mathsf x_{i\,j}  - \mathsf x_{(i+1)\,j}).
\end{equation}
Since $\L_{\det_{n\,k}}$ is a vector space, then~\eqref{lem:NKOddIJInL:eq1} and~\eqref{lem:NKOddIJInL:eq2} together imply that $(\mathsf x_{i\,j}  - \mathsf x_{(i+1)\,j}) \in \L_{\det_{n\,k}}$. 
\end{proof}

\begin{corollary}\label{cor:DimNKOddEqNK}Assume that  $n \ge k$,  $|\F| > k$   and  $n + k$ is even. Then
\begin{enumerate}[label=(\alph*), ref=(\alph*)]
\item\label{cor:DimNKOddEqNK:part1} $\L_{\det_{n\,k}} = \Span (\{\mathsf x_{i\,j}  - \mathsf x_{(i+1)\,j} \mid 1 \le i \le (n-1),\; 1 \le j \le k \})$;
\item\label{cor:DimNKOddEqNK:part2} $\dim(\L_{\det_{n\,k}}) = (n-1)k$.
\end{enumerate}
\end{corollary}

Here we provide another proof of~\cite[Lemma~3.7]{Guterman2025b} which employs the theory developed previously.

\begin{lemma}\label{lem:NKOddRadDetNK}Assume that $n \ge k$,  $|\F| > k$  and $n + k$ is odd. Then\linebreak
$\rad(\det_{n\,k}) =  W_{n\,k}.$
\end{lemma}
\begin{proof}On the one hand, $\dim(\L_{\det_{n\,k}}) + \dim(\rad(\det_{n\,k})) \le nk$ by Lemma~\ref{lem:dzeroeqk}\ref{lem:dzeroeqk:part1}. Since $\dim(\L_{\det_{n\,k}}) = (n-1)k$ by Corollary~\ref{cor:DimNKOddEqNK}\ref{cor:DimNKOddEqNK:part2}, 
\begin{equation}\label{lem:NKOddRadDetNK:eq1}
\dim(\rad(\det_{n\,k})) \le k.
\end{equation}
On the other hand,
\begin{equation}\label{lem:NKOddRadDetNK:eq2}
\rad(\det_{n\,k}) \supseteq W_{n\,k}
\end{equation}
by Corollary~\ref{cor:WNKIsInRadDetNK}. The definition of $W_{n\,k}$ implies that
\begin{equation}\label{lem:NKOddRadDetNK:eq4}
\dim(W_{n\,k}) = k.
\end{equation}
Hence,
\begin{equation}\label{lem:NKOddRadDetNK:eq3}
\dim(\rad(\det_{n\,k})) \overset{\eqref{lem:NKOddRadDetNK:eq2}}{\ge} \dim(W_{n\,k}) \overset{\eqref{lem:NKOddRadDetNK:eq4}}{=\joinrel=\joinrel=} k.
\end{equation}
By aligning together the inequalities~\eqref{lem:NKOddRadDetNK:eq1} and~\eqref{lem:NKOddRadDetNK:eq3} we conclude that\linebreak $\dim(\rad(\det_{n\,k})) = k$. Hence, $\rad(\det_{n\,k}) = W_{n\,k}$ by~\eqref{lem:NKOddRadDetNK:eq2} and~\eqref{lem:NKOddRadDetNK:eq4}.
\end{proof}

\begin{corollary}\label{cor:NKOddSatCond}Assume that  $n \ge k$  and  $n + k$ is odd. Then $\dim\left(\mathcal L_{\det_{n\,k}}\right) + \dim \left(\rad(\det_{n\,k})\right) = nk$.
\end{corollary}
%

\begin{theorem}\label{thm:homogenphipsiDetNKOdd}Assume that  $n \ge k + 2,$ $|\F| > k$, $k \ge 3$ and  $n + k$ is odd.  If  $\phi, \psi \colon \F^n \to \F^n$ are two maps satisfying the condition~(\ref{cond:DETNK}), then there exist $A \in \M_{n\, n}(\F)$ and $B \in \M_{k\, k}(\F)$ such that
\begin{equation}\label{thm:homogenphipsiDetNKOdd:eqq}
\det_{n\,k} \Bigl(A(|i_1,\ldots, i_k]\Bigr) \det_k \Bigl(B\Bigr) = (-1)^{i_1 + \ldots + i_k - 1 - \ldots - k}
\end{equation}
for all increasing sequences $1 \le i_1 < \ldots < i_k \le n $ and
\begin{equation}\label{thm:homogenphipsiDetNKOdd:eq}
\phi(X) = \psi(X) = AXB \pmod{W_{n\,k}}\qquad\mbox{for all}\;\;X \in \M_{n\,k}(\F).
\end{equation}
\end{theorem} 

\begin{proof}
The definition of $\det_{n\,k}$ implies that $\det_{n\,k}$ is a homogeneous polynomial of degree $k$. By Corollary~\ref{cor:NKEvenSatDimCond} we conclude that $\det_{n\,k}$, $\phi$ and $\psi$ satisfy the conditions of Theorem~\ref{thm:homogenphipsi}. Together with Lemma~\ref{lem:NKOddRadDetNK} this implies that there exists a linear map $T_{\rad (\det_{n\,k})} \colon \M_{n\,k}(\F)/W_{n\,k} \to \M_{n\,k}(\F)/W_{n\,k}$ preserving $\left(\det_{n\,k}\right)_{\rad (\det_{n\,k})}$ such that 
\begin{equation}\label{thm:homogenphipsiDetNKOdd:eq1}
\pi_{\rad (\det_{n\,k})} \circ \phi = \pi_{\rad (\det_{n\,k})} \circ \psi = T_{\rad (\det_{n\,k})} \circ \pi_{\rad (\det_{n\,k})}. 
\end{equation}

Let $\iota \colon \M_{n\,k}(\F)/W_{n\,k} \to \M_{n\,k}(\F)$ be any linear right inverse of $\pi_{\rad (\det_{n\,k})}$, that is, $\iota$ is a linear map satisfying the condition
\begin{equation}\label{thm:homogenphipsiDetNKOdd:eq2}
\pi_{\rad (\det_{n\,k})} \circ \iota = \id_{\M_{n\,k}(\F)/W_{n\,k}}.
\end{equation}
Definition~\ref{def:radicalfunc} implies that
\begin{multline}\label{thm:homogenphipsiDetNKOdd:eq3}
\det_{n\,k} \circ \iota = \left(\left(\det_{n\,k}\right)_{\rad (\det_{n\,k})} \circ \pi_{\rad (\det_{n\,k})}\right) \circ \iota\\
= \left(\det_{n\,k}\right)_{\rad (\det_{n\,k})} \circ \left(\pi_{\rad (\det_{n\,k})} \circ \iota\right)\overset{\eqref{thm:homogenphipsiDetNKOdd:eq2}}{=\joinrel=\joinrel=\joinrel=} \left(\det_{n\,k}\right)_{\rad (\det_{n\,k})}.
\end{multline}

Let $T$ be defined by 
\begin{equation}\label{thm:homogenphipsiDetNKOdd:eq4}
T = \iota \circ T_{\rad (\det_{n\,k})} \circ \pi_{\rad (\det_{n\,k})}.
\end{equation}
Then
\begin{multline*}\label{thm:homogenphipsiDetNKOdd:eq5}
\det_{n\,k} \circ T \overset{\eqref{thm:homogenphipsiDetNKOdd:eq4}}{=\joinrel=\joinrel=\joinrel=} \det_{n\,k} \circ \left(\iota \circ T_{\rad (\det_{n\,k})} \circ \pi_{\rad (\det_{n\,k})}\right)\\
= \left(\det_{n\,k} \circ \iota\right) \circ T_{\rad (\det_{n\,k})} \circ \pi_{\rad (\det_{n\,k})}\\
\overset{\eqref{thm:homogenphipsiDetNKOdd:eq3}}{=\joinrel=\joinrel=\joinrel=} \left(\det_{n\,k}\right)_{\rad (\det_{n\,k})} \circ T_{\rad (\det_{n\,k})} \circ \pi_{\rad (\det_{n\,k})}\\
 = \left(\left(\det_{n\,k}\right)_{\rad (\det_{n\,k})} \circ T_{\rad (\det_{n\,k})}\right) \circ \pi_{\rad (\det_{n\,k})}\\
 \overset{\scriptsize T_{\rad (\det_{n\,k})}\;\mbox{preserves}\;\left(\det_{n\,k}\right)_{\rad (\det_{n\,k})}}{=\joinrel=\joinrel=\joinrel=\joinrel=\joinrel=\joinrel=\joinrel=\joinrel=\joinrel=\joinrel=\joinrel=\joinrel=\joinrel=\joinrel=\joinrel=\joinrel=\joinrel=\joinrel=\joinrel=\joinrel=\joinrel=\joinrel=} \left(\det_{n\,k}\right)_{\rad (\det_{n\,k})} \circ \pi_{\rad (\det_{n\,k})} = \det_{n\,k}.
\end{multline*}
Thus, $T$ is a linear map preserving $\det_{n\,k}$. By Lemma~\ref{lem:THMNPLUSKODDKALL}, there exist $A \in \M_{n\, n}(\F)$ and $B \in \M_{k\, k}(\F)$ satisfying the condition~\eqref{thm:homogenphipsiDetNKOdd:eqq} and a linear map $\omega \colon \M_{n\,k}(\F) \to W_{n\,k}$ such that $T(X) = AXB + \omega(X)$ for all $X \in \M_{n\,k}(\F)$. This implies that

\begin{equation}\label{thm:homogenphipsiDetNKOdd:eq6}
T(X) = AXB + \omega(X)\;\;\mbox{for all}\;\;X \in \M_{n\,k}(\F).
\end{equation}
By applying $\pi_{\rad (\det_{n\,k})}$ to both the sides of~\eqref{thm:homogenphipsiDetNKOdd:eq6} we obtain that
\begin{equation}\label{thm:homogenphipsiDetNKOdd:eq7}
\left[\pi_{\rad (\det_{n\,k})} \circ  T\right](X) = \pi_{\rad (\det_{n\,k})}(AXB + \omega(X))\;\;\mbox{for all}\;\;X \in \M_{n\,k}(\F).
\end{equation}
Consider the right-hand side of~\eqref{thm:homogenphipsiDetNKOdd:eq7}. Since $\omega$ sends $ \M_{n\,k}(\F)$ to $W_{n\,k}$, then
\begin{equation}\label{thm:homogenphipsiDetNKOdd:eq8}
\pi_{\rad (\det_{n\,k})}(AXB + \omega(X)) = \pi_{\rad (\det_{n\,k})}(AXB)\;\;\mbox{for all}\;\;X \in \M_{n\,k}(\F).
\end{equation}
The left-hand side of~\eqref{thm:homogenphipsiDetNKOdd:eq7} is transformed as follows
\begin{multline}\label{thm:homogenphipsiDetNKOdd:eq9}
\pi_{\rad (\det_{n\,k})} \circ  T \overset{\eqref{thm:homogenphipsiDetNKOdd:eq4}}{=\joinrel=\joinrel=\joinrel=} \pi_{\rad (\det_{n\,k})} \circ \left(\iota \circ T_{\rad (\det_{n\,k})} \circ \pi_{\rad (\det_{n\,k})}\right)\\
= \left(\pi_{\rad (\det_{n\,k})} \circ \iota \right) \circ T_{\rad (\det_{n\,k})} \circ \pi_{\rad (\det_{n\,k})} \overset{\eqref{thm:homogenphipsiDetNKOdd:eq2}}{=\joinrel=\joinrel=\joinrel=} T_{\rad (\det_{n\,k})} \circ \pi_{\rad (\det_{n\,k})} 
\end{multline}
By substituting the equalities~\eqref{thm:homogenphipsiDetNKOdd:eq8} and~\eqref{thm:homogenphipsiDetNKOdd:eq9} into~\eqref{thm:homogenphipsiDetNKOdd:eq7} we obtain that
\begin{equation}\label{thm:homogenphipsiDetNKOdd:eq99}
\left[T_{\rad (\det_{n\,k})} \circ \pi_{\rad (\det_{n\,k})}\right](X) = \pi_{\rad (\det_{n\,k})}(AXB)\;\;\mbox{for all}\;\;X \in \M_{n\,k}(\F).
\end{equation}
Finally, substituting~\eqref{thm:homogenphipsiDetNKOdd:eq99} into~\eqref{thm:homogenphipsiDetNKOdd:eq1} yields
\begin{multline*}
\left[\pi_{\rad (\det_{n\,k})} \circ \phi\right](X) = \left[\pi_{\rad (\det_{n\,k})} \circ \psi\right](X)\\ = \pi_{\rad (\det_{n\,k})}(AXB)\;\;\mbox{for all}\;\;X \in \M_{n\,k}(\F).
\end{multline*}
This is equivalent to the condition~\eqref{thm:homogenphipsiDetNKOdd:eq}. Therefore, \eqref{thm:homogenphipsiDetNKOdd:eq} is satisfied.
\end{proof}

\section{Further work}
\label{sec:furthwork}

The author supposes that it is not possible to omit the conditions $\deg (P) > |\F|$ and $\dim (\mathcal L_P) + \dim (\rad(P)) = n$ in the statement of Theorem~\ref{thm:homogenphipsi}, but did not succeed to provide an example proving it. Any progress in this direction will be interesting.

When $\deg(P) < |\F|$, the main difficulty in constructing such example is that the condition 
\begin{equation}\label{eq:condfurthwork}
P(\mathbf{x} + \lambda \mathbf{y}) = P(\phi(\mathbf{x}) + \lambda \psi(\mathbf{y}))\;\;\mbox{for all}\;\;\mathbf{x},\mathbf{y} \in \mathbb \F^n,\; \lambda \in \F,
\end{equation}
provides $\deg(P)$ equations on $\phi$ and $\psi$ obtained by comparing coefficients at $\lambda^k$ on both the sides of the condition~\eqref{eq:condfurthwork}. From the author's experience, in most cases these equations imply that $\phi = \psi = \id$ and consequently $\psi$ and $\phi$ are linear. If $\deg(P) > |\F|$, then author tried the brute-force computer search, but it did not help.

Nevertheless, these conditions do not seem to be restrictive or difficult to verify, at least for the polynomial matrix invariants known to the author. The condition $\deg (P) < |\F|$ is frequently assumed when the characterisation of linear preservers is established. As the reader can observe in Section~\ref{sec:AppToCullisDet}, verifying the condition $\dim (\mathcal L_P) + \dim (\rad(P)) = n$ for $P = \det_{n\,k}$ (Definition~\ref{def:DETNK}) is a quite technical procedure. Moreover, the author notes that this procedure employs the same observations as the solution to the corresponding linear preserver problem. The same holds for the other polynomial matrix invariants discussed in Section~\ref{sec:intro}. 

%
%
%
%
%
%
%

\section{Declaration of competing interest}
We have no competing interest to declare.

\section{Data availability}
No data was used for the research described in the article.

\section*{Acknowledgements}

The author thanks his supervisor, Professor Alexander Guterman,  for pointing out the direction of research, constant support and for his comments on the draft. The research was supported by the scholarship of the Center for Absorption in Science, the Ministry for Absorption of Aliyah, the State of Israel.

\bibliographystyle{plain}
\bibliography{nonlinearmapspolynomial}

@book{Lang2002,
  title = {Algebra},
  ISBN = {9781461300410},
  ISSN = {2197-5612},
  url = {http://dx.doi.org/10.1007/978-1-4613-0041-0},
  DOI = {10.1007/978-1-4613-0041-0},
  journal = {Graduate Texts in Mathematics},
  publisher = {Springer New York},
  author = {Lang,  S.},
  year = {2002}
}

@article{Costara2021,
  title = {Nonlinear maps preserving the elementary symmetric functions},
  volume = {21},
  ISSN = {1855-3966},
  url = {http://dx.doi.org/10.26493/1855-3974.2488.f5c},
  DOI = {10.26493/1855-3974.2488.f5c},
  number = {1},
  journal = {Ars Mathematica Contemporanea},
  publisher = {University of Primorska Press},
  author = {Costara,  C.},
  year = {2021},
  pages = {P1.09}
}

@article{Dolinar2002,
  title = {Determinant preserving maps on matrix algebras},
  volume = {348},
  ISSN = {0024-3795},
  url = {http://dx.doi.org/10.1016/S0024-3795(01)00578-X},
  DOI = {10.1016/s0024-3795(01)00578-x},
  number = {1–3},
  journal = {Linear Algebra and its Applications},
  publisher = {Elsevier BV},
  author = {Dolinar,  G. and Šemrl,  P.},
  year = {2002},
  pages = {189–192}
}

@article{TAN2003311,
title = {On determinant preserver problems},
journal = {Linear Algebra and its Applications},
volume = {369},
pages = {311-317},
year = {2003},
issn = {0024-3795},
doi = {https://doi.org/10.1016/S0024-3795(02)00739-5},
url = {https://www.sciencedirect.com/science/article/pii/S0024379502007395},
author = {Tan V. and Wang F.}
}

@article{Kuzma2008,
  title = {A note on immanant preservers},
  volume = {155},
  ISSN = {1573-8795},
  url = {http://dx.doi.org/10.1007/s10958-008-9247-4},
  DOI = {10.1007/s10958-008-9247-4},
  number = {6},
  journal = {Journal of Mathematical Sciences},
  publisher = {Springer Science and Business Media LLC},
  author = {Kuzma,  B.},
  year = {2008},
  pages = {872–876}
}

@inbook{GF,
  author = {Frobenius, F. G.},
  title = {\"{U}ber die Darstellung der endlichen Gruppen durch lineare Substitutionen},
  publisher = {S. B. Deutsch. Akad. Wiss. Berlin},
  pages = {944--1015},
  year = 1897,
  doi = {10.3931/e-rara-18879},
}

@article{Waterhouse1983, author = {Waterhouse, W. C.}, title = {Invertibility 
of linear maps preserving matrix invariants}, journal = {Linear Multilinear 
Algebra}, volume = {13}, number = {2}, pages = {105--113}, year = {1983}, 
publisher = {Taylor \& Francis}, doi = {10.1080/03081088308817510}, 


optURL = { 
    
        https://doi.org/10.1080/03081088308817510
    
    

},
    
    

}

@misc{Guterman2025,
      title={{Linear maps preserving the Cullis' determinant. I}}, 
      author={Guterman, A. and Yurkov, A.},
      year={2025},
      eprint={2512.13437},
      archivePrefix={arXiv},
      primaryClass={math.CO},
      note={\emph{Available at} https://arxiv.org/abs/2512.13437}, 
}

@misc{Guterman2025b,
      title={{Linear maps preserving the Cullis' determinant. II}}, 
      author={Guterman, A. and Yurkov, A.},
      year={2025},
      eprint={2512.13445},
      archivePrefix={arXiv},
      primaryClass={math.CO},
      note={\emph{Available at} https://arxiv.org/abs/2512.13445},
}

@article{Guterman2025c, author = {Guterman, A. and Yurkov, A.}, 
 title = "{Linear maps preserving the Cullis'
determinant. III}", 
 year = 2025, 
journal = {Spec. {M}atrices}, 
 note = "\emph{(submitted for publication)}",
}

@article{NAKAGAMI2007422, title = {On {C}ullis’ determinant for rectangular 
matrices}, journal = {Linear Algebra Appl.}, volume = {422}, number = {2}, 
pages = {422--441}, year = {2007}, issn = {0024-3795}, doi = 
{10.1016/j.laa.2006.11.001}, author 
= {Nakagami, Y. and Yanai, H.} }

@book{cullis1913,
  title={Matrices and Determinoids: Volume 1},
  author={Cullis, C.~E.},
  series={Calcutta University Readership Lectures},
  year={1913},
  publisher={Cambridge University Press}
}

@article{LAMA199233,

title = {A survey of linear preserver problems},
journal = {Linear Multilinear Algebra},
author = {Pierce, S. and others},
volume = {33},
number = {1-2},
pages = {},
year  = {1992},
publisher = {Taylor & Francis},
doi = {10.1080/03081089208818176},

 }

@article{amiri2010, 
AUTHOR = {Amiri, A. and Fathy, M. and Bayat, M.}, 
     TITLE = {Generalization of some determinantal identities for non-square
              matrices based on {R}adic's definition},
   JOURNAL = {TWMS J. Pure Appl. Math.},
  FJOURNAL = {TWMS Journal of Pure and Applied Mathematics},
    VOLUME = {1},
      YEAR = {2010},
    NUMBER = {2},
     PAGES = {163--175},
      ISSN = {2076-2585,2219-1259},
   MRCLASS = {15A15},
  MRNUMBER = {2766622},
MRREVIEWER = {Christos\ Kravvaritis}, 
}

@article{Guterman11112009,
author = {A. Guterman and B. Kuzma},
title = {Preserving Zeros of a Polynomial},
journal = {Communications in Algebra},
volume = {37},
number = {11},
pages = {4038--4064},
year = {2009},
publisher = {Taylor \& Francis},
doi = {10.1080/00927870802545687},


URL = { 
    
        https://doi.org/10.1080/00927870802545687
    
    

},
eprint = { 
    
        https://doi.org/10.1080/00927870802545687
    
    

}

}

@article{SEMRL20081051,
title = {Commutativity preserving maps},
journal = {Linear Algebra and its Applications},
volume = {429},
number = {5},
pages = {1051-1070},
year = {2008},
issn = {0024-3795},
doi = {https://doi.org/10.1016/j.laa.2007.05.006},
url = {https://www.sciencedirect.com/science/article/pii/S0024379507002133},
author = {P. Šemrl}
}

\end{document}